\documentclass[12pt]{article}

\def \qed{\hfill$\square$}

\def\eps{{\varepsilon}}

\makeatother

\usepackage{xcolor}

\usepackage[margin=3cm]{geometry}

\usepackage[normalem]{ulem}
\usepackage{amsthm, bbm}
\usepackage{amsmath, amssymb}
\usepackage{mathrsfs}
\usepackage{hyperref}
\usepackage{graphicx, psfrag, epstopdf}
\graphicspath{{fig/}}
\usepackage{float}
\usepackage{wrapfig}
\usepackage{xcolor}
\usepackage{changepage}

\theoremstyle{definition}
\newtheorem{definition}{Definition}[section]
\newtheorem{theorem}[definition]{Theorem}
\newtheorem{lemma}[definition]{Lemma}

\newtheorem{corollary}[definition]{Corollary}
\newtheorem{proposition}[definition]{Proposition}

\newtheorem{remark}[definition]{Remark}
\newtheorem*{remark*}{Remark}

\numberwithin{equation}{section}

\definecolor{OliveGreen}{rgb}{0,0.6,0}

\def\R{\mathbb{R}}
\def\Z{\mathbb{Z}}
\def\T{\mathbb{T}}
\def\N{\mathbb{N}}

\newcommand{\lp}{log-Puiseux }
\newcommand{\mm}{{\mathcal M}}
\DeclareMathOperator{\intt}{int}

\theoremstyle{definition}
\newtheorem*{inner}{\innerheader}
\newcommand{\innerheader}{}
\newenvironment{namedtheorem}[1]
 {\renewcommand\innerheader{#1}\begin{inner}}
 {\end{inner}}

\title{Weak mixing for area preserving flows on surfaces}
\author{A. Kanigowski, A. Okunev, R. Zelada}
\date{}

\begin{document}
\maketitle
\begin{abstract}
Let $(\phi_t)$ be an area-preserving smooth flow on a compact, connected, orientable surface $\mathcal M$ with at least one but finitely many fixed points. Assume that $(\phi_t)$ is analytic (up to a canonical change of coordinates) in the neighborhood of each saddle fixed point. We show that the flow $(\phi_t)$ is weakly mixing on each of its (finitely many) quasi-minimal components.

 \end{abstract}
\tableofcontents
\section{Introduction}
Smooth area preserving flows on compact, connected orientable surfaces are one of the most basic classes of dynamical systems. They provide the lowest dimensional class of systems with non trivial ergodic, mixing, and spectral properties. These properties have been studied extensively in the last $50$ years. 
They differ substantially for flows with and without fixed points. If the flow has no fixed points, the Poincar\'e-Hopf index theorem implies that the surface is the $2$-torus and the flow is a {\em reparametrization} (or time change) of a linear flow. Kolmogorov, \cite{Kolg}, showed that under the additional assumption that the  frequency is diophantine, the reparametrization is smoothly conjugated to the original linear flow. On the other hand, Sklover, \cite{Sklo}, showed that there are Liouvillian frequencies for which the corresponding reparametrization is weakly mixing. It follows from a result by Katok, \cite{Kat}, that smooth flows (with no fixed points) on the torus are never mixing.

Mixing properties of flows on surfaces change in the presence of fixed points. Let us now assume that the set of fixed points is non-empty and finite. The two main cases one distinguishes here are non-degenerate and degenerate fixed points (depending on whether the Hessian at a fixed point vanishes or not). Historically, the first result on mixing of surface flows is due to Kochergin, \cite{Koch1},  who showed that on every compact, connected, orientable surface of genus $\geq 1$ there exists a smooth flow which is mixing (these flows are now called Kochergin flows). The examples in \cite{Koch1} were smooth flows with degenerate fixed points. The case of general degenerate saddles\footnote{Given an area-preserving smooth flow  $(\phi_t)$  on a compact, connected, orientable smooth surface $\mathcal M$, a point $p\in\mathcal M$ is called a \textit{saddle} if $p$ is fixed by $(\phi_t)$ and the local Hamiltonian defining $(\phi_t)$ near $p$ does not achieve a local maximum or minimum at $p$.} has not yet been systematically studied and in this paper we address this situation in the analytic case. 

In the case where all the fixed points are non-degenerate, the situation  has been widely studied 
and shown to be quite delicate. First, it was shown by Khanin and Sinai, \cite{KS}, that a flow on $\T^2$  with one homoclinic loop (see Figure~\ref{Fig1}) is mixing (on the transitive component) for almost every frequency $\alpha$. This result was improved by Kochergin to all irrational $\alpha$, \cite{Koch2}. When the genus is $\geq 2$ and all the fixed points are non-degenerate, the main mixing results for ``typical'' flows were obtained by Ulcigrai. Indeed, in \cite{Ulc1} Ulcigrai showed that almost every\footnote{As we explain later, these flows (on quasi-minimal components) can be represented as special flows over IET's (Poincar\'e first return map) with roof function (Poincar\'e first return time) that is smooth except at finitely many points. Almost every refers to a.e. IET with respect to Lebesgue measure.} flow with so called symmetric singularities is not mixing. It was shown in \cite{Ulc} that almost every such flow is weakly mixing.

\begin{wrapfigure}{r}{0.4\textwidth}
 \centering
 \resizebox{!}{5.4cm} {\includegraphics[angle=1]{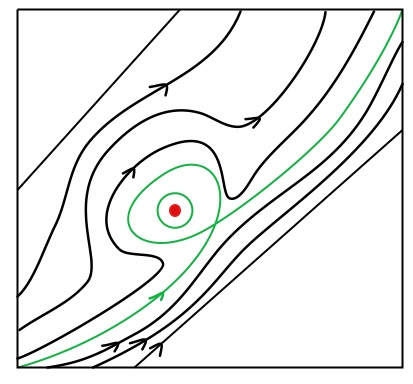}}
  \caption{\small Homoclinic loop of a locally Hamiltonian flow. Figure taken from \cite{FKZ}.}\label{Fig1}
 \label{FigSaddle}
\end{wrapfigure} 
If the singularities are asymmetric, then  it was shown in \cite{Ulc2} that almost every such flow (with one singularity) is mixing. This was improved by Ravotti, \cite{Ravotti}, who showed mixing for almost every such flow with any number of singularities (at a subset of fixed points).
On the other hand, minimal smooth flows that are mixing were first constructed by Chaika-Wright, \cite{ChW}, where the authors showed existence of a flow with symmetric singularities on a genus $5$ surface which is mixing complementing the result of Ulcigrai, \cite{Ulc1}. Recently in \cite{FKZ} Fayad and the first and third author constructed a smooth flow with asymmetric singularities on a genus $2$ surface which is not mixing  complementing the resullts  in \cite{Ulc2},\cite{Ravotti}. Some finer ergodic, mixing, and spectral properties of smooth flows on surfaces were studied in \cite{FKU,KLU,FFK,FaKa}.

One of the main tools to study smooth flows on surfaces and also the related translation flows is {\em renormalization}, i.e. a rich dynamics on the space of all $d$-IET's (which appear as Poincar\'e section for smooth surface flows). This is the main reason why most of the results are proven for almost every flow: one imposes some good behavior of the orbit under renormalization  of an arbitrary $d$-IET which allows to prove a given result and then shows that this good behavior holds on a full measure set of $d$-IET's. 
Using renormalization, Avila and Forni, \cite{AF}, showed that almost every translation flow is weakly mixing (this also holds for a.e. IET).
On the other hand there are very few non-trivial results that hold for {\em every} smooth (or translation) flow. One of the very few theorems of this kind is due to Katok, \cite{Kat}, who showed that no translation flow is mixing. Our main result is a statement about \emph{every smooth surface flow} (under the extra assumption of local analyticity near saddles). We say that a flow is {\em analytic} at a point $p$ if there exists a neighborhood of $p$ and an area-preserving $C^\infty$ change of coordinates for which the flow is analytic in these new coordinates.

\begin{namedtheorem}{Theorem A}
Let $(\phi_t)_{t\in\R}$ be a $C^\infty$ area-preserving flow, which has at least one, but finitely many fixed points, and is analytic at each saddle fixed point. Then $(\phi_t)_{t\in\R}$ is weakly mixing on any of its quasi-minimal components.
\end{namedtheorem}

\noindent In the case where all fixed points are non-degenerate, the above result improves \cite{Ulc} from almost every flow to every flow.\footnote{If the flow is only $C^\infty$ but all its saddles are nondegenerate, Theorem~A$^\prime$ below can still be applied. Indeed, it follows from Morse lemma that every non-degenerate saddle has logarithmic passing time.} The above result is also the first systematic result in the case of degenerate fixed points as Kochergin, \cite{Koch1}, only  studied a class of examples. In fact, using analyticity we are able to estimate the first return time and show that it is always logarithmic, power or a product of those.

As we will see, Theorem A is a direct consequence of the following result for special flows.
\begin{namedtheorem}{Theorem A$^\prime$}
Consider the special flow $T^f$ over a right-continuous IET $T:[0,1)\rightarrow [0,1)$ and under a roof function $f$. 
Assume that there exists a finite set $\mathcal A_f \subset (0, 1)$ such that $f$ is $C^2$ on $(0, 1) \setminus \mathcal A_f$, and $f$ has an infinite singularity at each $x \in \mathcal {A}_f$: 
$\lim_{y \to x^-} f(y) = \lim_{y \to x^+} f(y) = \infty$. 
Suppose that additionally
\begin{enumerate}
  \renewcommand{\labelenumi}{(\roman{enumi})}
    \item Every discontinuity of $T$ belongs to $\mathcal A_f$.

    \item  $f$ is of \textit{at least logarithmic growth}: there are $B,C>0$ such that 
    \begin{equation} \label{e:log-growth-intro}
    f''(x)\geq \frac{C}{\min_{y\in \mathcal A_f}|x-y|^2}-B
    \end{equation}
    for every $x\in (0,1)\setminus\mathcal A_f$. 
\end{enumerate}
Then $T$ is ergodic if and only if $T^f$ is weak mixing.  
\end{namedtheorem}
\begin{remark}
In Section \ref{s:analytic-singularities}, we show that if $(\phi_t)$ is analytic near all its saddle fixed points (and has at least one fixed point), then on each quasi-minimal component the return time function satisfies \eqref{e:log-growth-intro}. Thus, Theorem A follows from Theorem A$^\prime$.
\end{remark}

\paragraph{Outline of the proof of Theorem A.}
The proof of Theorem A consists of two main parts. In the first part (Sections 4-5), we reduce Theorem A to Theorem A$^\prime$.
By a classical decomposition result~\cite{Mai, Kat73, NZ}, the surface $\mathcal M$ splits into finitely many periodic and quasi-minimal components. 
Consider now the restriction $\phi_t|_{\mm_i}$ of the flow $(\phi_t)$ to any of the quasi-minimal components $\mm_i$. 
We use the standard technique called \emph{special flow representation:} $\phi_t|_{\mm_i}$ can be represented (i.e. there exists a measurable conjugacy) as a special flow $T^f$ over an IET $T$ and under a roof function $f$ that is smooth everywhere in the domain of $T$ except for finitely many points.
While it is well-known that $\phi_t|_{\mm_i}$ can be represented as a special flow over an IET, we couldn’t find this fact in the literature in the required generality. We provide a full proof of this fact  in Section~\ref{s:flows2} and, additionally, show that one can pick  $T^f$ to satisfy property~\emph{(i)} from Theorem~A$^\prime$ (an observation which we could also not find in the literature). We then use resolution of analytic singularities~\cite{Hir} to obtain the asymptotics of the time it takes for a trajectory of $(\phi_t)$ to pass a small neighborhood of a saddle (Section~\ref{s:analytic-singularities}). This proves property~\emph{(ii)}, and thus Theorem~A is reduced to Theorem~A$^\prime$.

In the second part of the proof (Sections 6-7), our goal is to show that $\phi_t|_{\mm_i}$  possesses an  abundant class of ``nice" special flow representations which, so to say, ``prevents" the existence of non-trivial $\phi_t|_{\mm_i}$-eigenfunctions. Existence of such an abundant class for a \textit{typical} area-preserving smooth flow was established in \cite{Ulc} by employing renormalization. However, as we mentioned above,  a similar result cannot be achieved (or, at the very least, easily achieved) for \textit{every} area-preserving smooth flow by simply employing renormalization.
The key observation to overcome this difficulty is to note that, roughly speaking, if every discontinuity of $T_\mathcal I$, the IET induced by $T$ on some subinterval $\mathcal I$ of the domain of $T$, has the property that its $T$-forward orbit goes through a discontinuity of $T$ before returning to $\mathcal I$, then the special flow induced by $T^f$ on $\mathcal I$ (which also is a special flow representation of  $\phi_t|_{\mm_i}$) is nice (i.e. inherits the properties~\emph{(i)–(ii)} of $T^f$ mentioned in Theorem A$^\prime$). Combining this observation with the classical fact that the forward orbits of the discontinuities of an aperiodic IET are dense in the domain of the IET, one obtains that there exist arbitrarily fine partitions of the domain of $T$ such that the special flow induced by $T^f$ in any of the members of any of these partitions satisfies~\emph{(i)–(ii)}.  One can then show 
 that the class formed by the special flow representations of $\phi_t|_{\mm_i}$ satisfying (i) and (ii) does indeed "prevent" the existence of non-trivial $\phi_t|_{\mm_i}$-eigenfunctions.

\section{ Background on interval exchange transformations}
In this section we review various results dealing with interval exchange transformations (IETs) needed for the arguments employed in the coming sections.\\
A  (right-continuous) interval exchange transformation $T:[a,b)\rightarrow [a,b)$, $a<b$, is a bijective map for which there exist an $N\geq 1$ and numbers $a_0=a<a_1<\cdots<a_N=b$ and $\sigma_1,...,\sigma_N\in \R$ such that for any $x\in [a_{i-1},a_{i})$, $Tx=x+\sigma_{i}$ for each $i\in\{1,...,N\}$.
Observe that a map $T$ defined in this way has at most $N-1$ discontinuities and preserves the Lebesgue measure.\\
For any right closed and left open interval $\Delta\subseteq [a,b)$, we let
$h_{\Delta,T}:\Delta\rightarrow\N$ be defined by
\begin{equation}\label{eq:FirstReturnTime}
    h_{\Delta,T}(x)=\inf\{n\in\N=\{1,2,...\}\,|\,T^nx\in \Delta\}.
\end{equation}
(Note that because $T$ preserves the Lebesgue measure, the  Poincar{\'e} recurrence Theorem and the fact that $T$ is right-continuous imply that  $h_{\Delta,T}$ is well-defined.)\\
Denote by $T_\Delta$ the first return map induced by $T$ on $\Delta$. In other words, for each $x\in \Delta$, $T_\Delta x=T^{h_{\Delta,T}(x)}x$. We denote the set of all the discontinuities of $T_\Delta$ by $\chi_\Delta$.
\subsection{Induced IETs and their discontinuities}
The following observation due to Katok~\cite[Proof of Lemma 2]{Kat} is needed for the proof of Theorems \ref{thm:RegularDisc}  below. It provides an upper bound for the number of discontinuities of an induced IET (see \cite[p. 128]{CornfeldErgodicTheory} for the correct upper bound) and, perhaps more importantly, it provides a qualitative description of the relationship between an IET and an IET induced by it.
\begin{proposition}\label{prop:TypesOfDiscontinuities}
    Let $N>1$, let $T:[0,1)\rightarrow[0,1)$ be a right-continuous IET with $N-1$ discontinuities, let $\Delta=[a,b)\subseteq [0,1)$, and let  $T_\Delta$ be the map induced by $T$ on $\Delta$. If $\alpha\in \chi_\Delta$ is a discontinuity,  then (a) $T^i\alpha\not\in [a,b)$ for $i\in\{1,...,h_{\Delta,T}(\alpha)-1\}$, (b) $T^{h_{\Delta,T}(\alpha)}\alpha=T_\Delta \alpha$, and (c) at least one of the following holds:
    \begin{enumerate}
    \item [(c.1)] For some $i\in\{0,...,h_{\Delta,T}(\alpha)-1\}$, $T^i\alpha$ is a discontinuity of $T$.
    \item [(c.2)] $T^{h_{\Delta,T}(\alpha)}\alpha=a$.
    \item [(c.3)] There is an  $i\in\{1,...,h_{\Delta,T}(\alpha)-1\}$, such that $\lim_{x\rightarrow \alpha^-}T^ix=b$.
    \end{enumerate}
\end{proposition}
\begin{remark}
    In light of Proposition \ref{prop:TypesOfDiscontinuities}, we have that if an  IET $T:[0,1)\rightarrow [0,1)$ has $N-1$ discontinuities, then for any $\Delta\subseteq [0,1)$, $T_\Delta$ has at most $N+1$ discontinuities. Equivalently, if $T$ "exchanges" $N$ intervals, $T_\Delta$ exchanges at most $N+2$ intervals.
\end{remark}
\subsection{A sufficient condition for minimality}
A (right-continuous) IET $T:[a,b)\rightarrow [a,b)$ is called \textit{minimal} (or quasi-minimal
to avoid confusion with the concept steming from topological dynamics) if for every $x\in[a,b)$, the set $\{T^nx\,|\,n\in\N\}$ is dense in $[a,b)$. The following proposition states that Keane's Condition (also known as the "disjoint infinite orbits condition") is sufficient to guarantee the minimality of $T$. As we will see, this condition arises naturally when studying the quasi-minimal components of a smooth-area preserving flow. 
\begin{proposition}[Keane's Condition, one-sided variant]\label{thm:KeanesCondition}
    Let $N>1$, let  $T:[0,1)\rightarrow[0,1)$ be a right-continuous IET, and let $a_1<\cdots<a_{N-1}$ be a list of the discontinuities of $T$ (so $0<a_1$ and $a_{N-1}<1$). Suppose that the sets $\Omega_j=\{T^na_j\,|\,n\in\N\}$, $j\in\{1,...,N-1\}$, are pairwise disjoint and that for each $j\in\{1,...,N-1\}$, $\Omega_j$ is infinite. Then, $T$ is quasi-minimal.
\end{proposition}
For the proof of Proposition \ref{thm:KeanesCondition}, we refer the reader to Proposition 4.1 in \cite{viana2006} (see also \cite[Theorem 5.4.1]{CornfeldErgodicTheory}).

\section{Background on smooth area preserving flows}\label{s:flows1}
Throughout this section we will let $\mathcal M$ denote a compact, connected and orientable smooth surface. Let $X:\mathcal M\rightarrow T\mathcal M$ denote a $C^\infty$ smooth tangent vector field with finitely many  or no zeroes, and let $(\phi_t)$, $t \in \mathbb R$, be the corresponding smooth flow.
We assume that $\phi_t$ preserves a smooth nondegenerate area form $\omega$.
Then the flow is \emph{locally Hamiltonian} (a proof is given, e.g., in~\cite{ChFKU}): the pair $(\phi_t)$ and $\omega$ determines a unique smooth closed
differential 1-form $\eta$ with the following properties:
\begin{enumerate}
    \item Locally, $\eta=dH$ for some smooth, real-valued, local function $H$,  called the \textit{local Hamiltonian}. 
    \item $\eta$ is the $X$-contraction of $\omega$: $\eta(Y)=\omega(X, Y)$ for any vector field $Y$. 
\end{enumerate}
Conversely, any smooth closed real-valued
differential 1-form $\eta'$ on $(\mathcal M,\omega)$ determines an  $\omega$-preserving flow known as the \textit{local Hamiltonian flow} associated to $\eta'$.
Selecting local coordinates $(x,y)$ such that $\omega = dx \wedge dy$ (they always exist), the vector field $X$ giving $(\phi_t)$ writes in terms of $H$ (the local Hamiltonian which satisfies $dH=\eta'$) by the usual Hamilton's equations:
\begin{equation} \label{e:hamilton-equations}
\left( \frac{\partial H}{\partial y}, -\frac{\partial H}{\partial x} \right).     
\end{equation}
Locally Hamiltonian flows are often called \emph{multi-valued Hamiltonian}, as one can treat $H$ as a globally defined multivalued function on $\mm$. 

Let us now restate some standard definitions in order to be precise.
For any $\mathcal I\subseteq \R$ and any $p\in\mathcal M$, we set  
$$
\phi_\mathcal I(p):=\{\phi_t(p)\,|\,t\in\mathcal I\}.
$$
For a given point $p\in\mathcal M$, we call the set $\phi_\R(p)$ its \emph{trajectory}, the set $\phi_{[0,\infty)}(p)$ its \emph{forward semi-trajectory}, and the set $\phi_{(-\infty,0]}(p)$ its \emph{backward semi-trajectory}.
A point $x \in \mm$ is a \emph{fixed point}, if $X(x)=0$; we denote the set of all fixed points by $\mm_*$.
Noting that
$dH(X) = \eta(X)=\omega(X, X)=0$, we see that $H$ is conserved by $(\phi_t)$.
Moreover, when restricted to the complement of $\mm \setminus \mm_*$, trajectories of $(\phi_t)$ locally coincide with connected components of the level sets of $H$.
We say that a \emph{stable separatrix} is a trajectory $\phi_{\mathbb R}(x_0)$, $x_0 \in \mm \setminus \mm_*$ that ends at a fixed point: $\lim_{t \to \infty}\phi_t(x_0)$ exists and is a fixed point.
A trajectory that begins at a fixed point, i.e. $\lim_{t \to -\infty}\phi_t(x_0)$ is a fixed point, is called an \emph{unstable separatrix}.
A trajectory can be both a stable and an unstable separatrix, then we call it a \emph{separatrix connection} between the fixed points it begins and ends at.

When $(\phi_t)$  has finitely many fixed points and separatrices, one has the following classical decomposition result.
\begin{theorem}[\cite{Mai, Kat73, NZ}] \label{t:decomposition}
    The surface $\mm$ can be decomposed into a finite union of $(\phi_t)$-invariant closed domains $\mm_i$ with nonempty disjoint interiors. The boundary of each $\mm_i$ is formed by polycycles (closed curves formed by fixed points and separatrix connections between them). The dynamics on each $\mm_i$ can only be:
    \begin{enumerate}
    \item $\mm_i$ is a \emph{quasi-minimal component}: any forward semi-trajectory in $\mm_i$ that is not a stable separatrix and not a fixed point is dense in $\mm_i$; any backward semi-trajectory in $\mm_i$ that is not an unstable separatrix and not a fixed point is dense in $\mm_i$.
    \item $\mm_i$ is an annulus filled by periodic orbits, or a disc filled by periodic orbits winding around a unique fixed point.
    \item $\mm_i$ is the whole $\mm$. Then either $\mm = S^2$ is the sphere with equatorial flow\footnote{i.e., the flow has two fixed points of the center type, and the punctured sphere between them is filled by periodic points}, or $\mm = \mathbb T^2$ is the torus filled by periodic orbits.
    \end{enumerate}
\end{theorem}
For the convenience of the reader, we provide a self-contained sketch of the proof of Theorem \ref{t:decomposition} in the next section. 
\begin{remark}
Maier's classical result (\cite{Mai}, see also the book~\cite{NZ}) actually describes nontrivial recurrent trajectories of non-area preserving flows as well. He also proved that the number of quasi-minimal components (or their analogues for non-area preserving case -- see~\cite{Mai, NZ}) is always bounded by the genus $g$ of $\mathcal M$. 
\end{remark}

\section{Special flow representations of smooth area-preserving flows} \label{s:flows2}
As before, throughout this section we will let $\mathcal M$ denote a compact, connected and orientable smooth surface and we will let $(\phi_t)$ denote a smooth-area preserving flow defined on $\mathcal M$ having at most finitely many (possibly zero) fixed points and a finite number of separatrices.

A standard tool used to study  dynamics on each quasi-minimal component of $(\phi_t)$ is to represent it by a special flow over an IET. This has been proven useful when studying ergodic properties of area-preserving surface flows -- see, for example,~\cite{ChW, Ravotti, Ulc1}.
To prove Theorem~A, we need the following additional property that holds if $\mm$ has at least one fixed point: one can choose the special flow representation in such a way that the roof function has two-sided infinite singularities at each point where it is undefined, and, additionally, it has a singularity  at every discontinuity of the IET.
\begin{theorem}\label{thm:AnalyticFlowRepresentation}
    Suppose that $\mm$ has at least one fixed point and let 
    $\mm_i$ be one of its  quasi-minimal components.
    Then the restriction of $(\phi_t)$ to $\mm_i$ can be represented (i.e. is measure theoretically isomorphic) as a special flow over a minimal IET $T:[a,b) \mapsto [a,b)$ and under a roof function $f$ that possess the following properties:
    \begin{enumerate}
    \item The function $f$ is defined at all but finitely many $x\in[a,b)$ and it is smooth at any $x \in (a,b)$ where it is defined. Furthermore, if $f$ is defined at $x$, one has that $Tx=\phi_{f(x)}(x)$.

    \item If $f$ is not defined at $x\in (a,b)$ it has a singularity there\footnote{Moreover, there is a separatrix paths (see Section~\ref{s:orthogonal} for the definition) of $(\phi_t)$ connecting $x$ with $Tx$ and another (potentially equal) separatrix path connecting $x$ with $\lim_{y\rightarrow x^-}Ty$.}:
    $$\lim_{y\rightarrow x^-}f(y)=\lim_{y\rightarrow x^+}f(y)=\infty.$$
    
    \item $f$ has at least one singularity on $(a,b)$, and it has a singularity at each of the discontinuities of $T$.

    \item \label{prop:endpoints}  For every integer $n \ge 0$, both $\lim_{x\rightarrow a^+}f^{(n)}(x)$ and $\lim_{x\rightarrow b^-}f^{(n)}(x)$ exist. 
    \end{enumerate}
\end{theorem}

\begin{figure}[ht]
\centering
\resizebox{!}{4cm} {\includegraphics{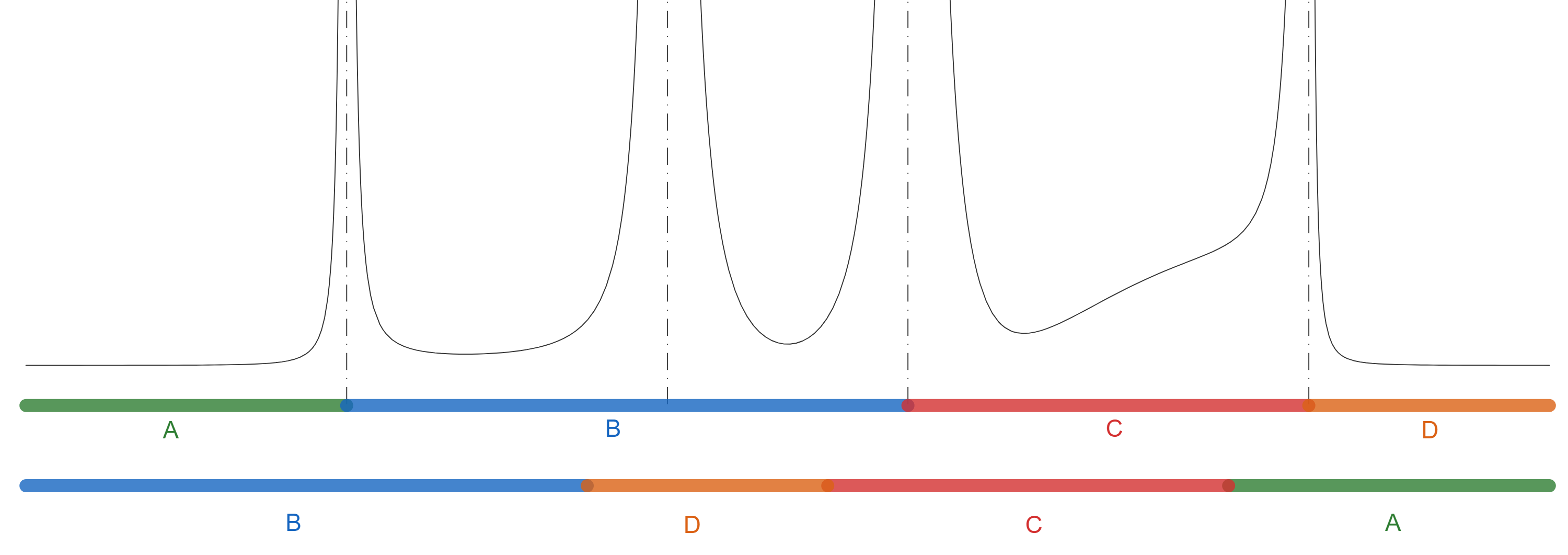}}
\caption{\small An example of a special flow satisfying the conclusions of Theorem~\ref{thm:AnalyticFlowRepresentation}.}
\label{Fig:ExampleSpecialFlowRep}
\end{figure}
Although existence of a special flow representation for smooth area-preserving flows is well known to experts, we could not find it stated in the literature in the required generality. Additionally, we have not found the properties of the roof function $f$ claimed in~Theorem \ref{thm:AnalyticFlowRepresentation}, which are crucial for our proof of Theorem~A.
For the reader's convenience, we present in this section a full proof of Theorem~\ref{thm:AnalyticFlowRepresentation}. 
The tools developed to prove Theorem~\ref{thm:AnalyticFlowRepresentation} will also allow us to give a brief sketch of a proof of Theorem~\ref{t:decomposition}.

\subsection{The orthogonal flow} \label{s:orthogonal}

In this subsection, we introduce a flow $(\psi_s)$, which is a useful tool for the study of the original flow $(\phi_t)$, and state an important technical proposition required for the sequel. 
Equip $\mm$ with a Riemannian metric consistent with the volume form $\omega$ that defines the symplectic structure.
\begin{lemma}
There exists a unique flow $(\psi_s)$ on $\mm \setminus \mm_*$ such that its trajectories are orthogonal to those of $(\phi_t)$, and for any local Hamiltonian $H$ of $(\phi_t)$ we have $\frac{dH}{ds} = 1$ along $(\psi_s)$.
\end{lemma}
\begin{proof}
Let $X$ be the vector field giving the flow $(\phi_t)$. Because $\omega$ is non-degenerated, there exists
a unique vector field  $Y$ defined on $\mm \setminus \mm_*$ which is orthogonal to $X$ and which satisfies the condition $\omega(X,Y)=1$. Let $\psi_s$ be the flow associated to $Y$.  Recall that $H$ is a local Hamiltonian if and only if $\omega(X, \cdot) = dH$. Thus, the condition $\omega(X,Y)=1$ rewrites as $\frac{dH}{ds} = 1$ along $\psi_s$.
\end{proof}

If $\gamma \subset \mm \setminus \mm_*$ is a trajectory segment (meaning that $\gamma = \phi_{[0, T]}(x_0)$ for some $T>0$ and $x_0 \in \mm \setminus \mm_*$) and $s$ is such that for every $s'\in [0,s]$,  $\psi_{s'}(\gamma) \subset \mm \setminus \mm_*$  is well defined, then $\psi_s(\gamma)$ must also be a trajectory segment. Indeed, $\gamma$ locally is a level set of a local Hamiltonian isolated from its critical points, and so is $\psi_s(\gamma)$ due to $\frac{dH}{ds} = 1$.
We will now state a proposition describing how $\psi_s(\gamma)$ degenerates when it approaches $\mm_*$, which is the main technical tool employed in the proof of Theorem~\ref{thm:AnalyticFlowRepresentation}.

Recall that a fixed point $x \in \mm_*$ is a \emph{saddle} if $H$ does not achieve a local maximum or minimum at it. This allows for saddles with degenerate linear part, i.e. such that $H$ has zero Hessian determinant. Note that a stable separatrix can only end at a saddle, and an unstable separatrix can only begin at a saddle.
We will say that a \emph{separatrix path} is a union of a beginning point $a \in \mm$, saddles $s_1, \dots, s_N \in \mm_*$, $N > 0$, an ending point $b \in \mm$, and the following (pieces of) separatrices connecting them:
a stable separatrix piece $\phi_{[0, \infty)}(a)$ connecting $a$ with $s_1$, $N-1$ separatrix connections\footnote{If $N=1$, there are no separatrix connections.} beginning at $s_i$ and ending at $s_{i+1}$ for $i=1, \dots, N-1$,
and an unstable separatrix piece $\phi_{(-\infty, 0]}(b)$ connecting $s_N$ with $b$.
We may have $a=s_1$ or $b=s_N$ here, and some of the saddles $s_i$ may coincide. So, technically, even a single saddle is a separatrix path for us.

\begin{proposition} \label{p:psi-s-star}
Let $\gamma \subset \mm \setminus \mm_*$ be a trajectory segment of $(\phi_t)$: $\gamma = \phi_{[0, T]}(x_0)$ for some $T>0$ and $x_0 \in \mm \setminus \mm_*$. 
Then either $\psi_s(\gamma) \subset \mm \setminus \mm_*$ is defined for all $s>0$, or there exists $s_* > 0$ such that $\psi_s(\gamma) \subset \mm \setminus \mm_*$ is well defined for all $s \in [0, s_*)$, and the limit with respect to Hausdorff distance $\lim_{s \to s_*^-} \psi_{s}(\gamma)$ exists and is either a separatrix path of $(\phi_t)$, as defined above, or a single fixed point.
\end{proposition}
\noindent The proposition covers the case $s > 0$, but a symmetric statement for $s<0$ also holds. 

A formal proof of this proposition is quite technical, so we present it in Appendix~\ref{a:limit-trajectory}.
However, in the special case when the flow is analytic, the proposition is heuristically clear in terms of how level sets $\{H(x, y) = h\}$ behave when $h$ approaches to a critical value.  

\subsection{First return maps are IETs} \label{s:IET-first-return}
In this subsection, we describe first return maps of area-preserving flows with zero or  finitely many fixed points and separatrices which are defined on closed, connected, orientable surfaces.
All transversals considered will be closed connected smooth curves with a finite length.
If $\Gamma$ is such transversal, one can always extend a local Hamiltonian $H$ to cover the whole transversal $\Gamma$, and, due to transversality, $H$ then provides a well-defined parametrization of $\Gamma$. We will then write
$\Gamma \cong [a, b]$, meaning that $\Gamma$ is parametrized by $H$, with the values of $H$ on $\Gamma$ changing from $a$ to $b$.
For a transversal $\Gamma \cong [a, b]$ as above, we denote by $\mathcal D \subset \Gamma$ the set formed by all $x$ such that one of the following holds
\begin{enumerate}
    \item $x$ is an endpoint of $\Gamma$

    \item $\phi_{(0, \infty)}(x)$ intersects $\Gamma$, and the first such intersection is an endpoint of $\Gamma$ 

    \item $\phi_{(0, \infty)}(x)$ does not intersect $\Gamma$ and forms a piece of a stable separatrix ending at a saddle.  
\end{enumerate}
This set is finite: the first criterion gives two points, the second criterion gives at most two points (as can be verified by considering backward trajectories of the endpoints of $\Gamma$), and the number of points satisfying the last criterion is bounded by the total number of separatrices of $(\phi_t)$. 

\begin{proposition} \label{p:first-return-IET}
    Let $\Gamma \cong [a, b]$ be a transversal, $P$ be the first return map induced by $(\phi_t)$ on $\Gamma$, and $\mathcal D \subset \Gamma$ be the finite set defined above. Then, there exists a right-continuous IET $T: [a, b) \mapsto [a, b)$ such that for each $x \in [a, b) \setminus \mathcal D$ its first return point $P(x)$  is well defined and coincides with $T(x)$. Moreover, every discontinuity of $T$ belongs to $\mathcal D$ and  if $P$ is not defined at some $x \in \mathcal D \cap [a, b)$, there exist  a separatrix path connecting $x$ with $T(x)$ and a separatrix path connecting $x$ with $\lim_{y\rightarrow x^-}Ty$ (provided that $x>a$).
\end{proposition}

Let us now state a sufficient criterion for this IET $T$ to be minimal.
\begin{proposition} \label{p:flows-Keane}
Let $\Gamma \cong [a, b]$ be a transversal such that:
\begin{enumerate}
    \item $\Gamma$ is disjoint from all separatrix connections of $(\phi_t)$;

    \item The endpoints of $\Gamma$ lie on a single $(\phi_t)$-trajectory, which is neither a stable separatrix nor a periodic orbit. The segment of this $(\phi_t)$-trajectory connecting the endpoints of $\Gamma$ only intersects $\Gamma$ at its endpoints. If this $(\phi_t)$-trajectory is an unstable separatrix, then it first intersects $\Gamma$ at one of its endpoints.
\end{enumerate}
    Then, the IET $T$ from Proposition~\ref{p:first-return-IET} has at least one discontinuity and satisfies Keane's condition for minimality, as stated in Proposition~\ref{thm:KeanesCondition}.
\end{proposition}

The rest of this subsection is a proof of these two propositions. We need some notation and lemmas first. 
For $\Gamma$ as above, say that a \emph{ $\Gamma$-arc} is a trajectory segment $A = \phi_{[0, T]}(x_0)$ that begins and ends on $\Gamma$ and is disjoint from both endpoints of $\Gamma$.
For each point $x \in \Gamma$ there is at most one $\Gamma$-arc starting there; denote it by $A(x)$.
Say that $x \in \Gamma$ is \emph{good} if $A(x)$ exists.
The set of good points is open in the topology of $\Gamma$ and does not contain the endpoints of $\Gamma$. Thus, it is a disjoint union of maximal open intervals, which we call \emph{good intervals}.
\begin{lemma} \label{l:good-intervals}
Let $(c, d) \subset \Gamma$ be a good interval. Then there exist one-sided limits 
$\lim_{x \to c^+}  A(x)$ and $\lim_{x \to d^-}  A(x)$
with respect to the Hausdorff distance.
Moreover, each of these limits is either
\begin{enumerate}
    \item a separatrix path\footnote{As $\Gamma$ is a transversal, this separatrix path does not begin or end at a fixed point.} starting and ending on $\Gamma$
    \item or a trajectory segment starting and ending on $\Gamma$ and containing at least one endpoint of $\Gamma$.
\end{enumerate}
\end{lemma}
\begin{proof}
Suppose for a moment that $\Gamma$ is a trajectory of the orthogonal flow $(\psi_s)$.
Take any $x_0 \in (c, d)$ and let $A_0 = A(x_0)$ be its $\Gamma$-arc.
Then other $\Gamma$-arcs $A(x)$, $x \in (c, d)$, are images of $A_0$ under the orthogonal flow: $A(x) = \psi_{x-x_0}(A_0) \subset \mm \setminus \mm_*$.  
Consider now the endpoint $d$ of the good interval, the other endpoint is treated symmetrically. There are two possibilities:

\noindent \emph{Case 1:} $\psi_{d-x_0}(A_0) \subset \mm \setminus \mm_*$ is well defined. Then this set is a segment of a $\phi_t$-trajectory that begins and ends on $\Gamma$.   
As this trajectory segment is not a $\Gamma$-arc, it must contain an endpoint of $\Gamma$.  

\noindent \emph{Case 2:} $\psi_{d-x_0}(A_0) \subset \mm \setminus \mm_*$ is not well defined. Then by Proposition~\ref{p:psi-s-star} there is a limit 
$\lim_{x \to d^{-}} A(x)$, which is a separatrix path or a fixed point of $(\phi_t)$. As this limit must begin and end on $\Gamma$, we conclude that it is a separatrix path.

For the general case we simply note that by performing small deformations of $\Gamma$ along the trajectories of $(\phi_t)$ going through $\Gamma$, one can assume without loss of generality that at every endpoint of a good interval $(c,d)$, there is a trajectory segment of $\psi_s$ which is a subinterval of $\Gamma\cap [c,d]$ containing either $c$ or $d$.
\end{proof}

\begin{lemma} \label{l:the-set-D}
    The set of all good points on $\Gamma$ coincides with $\Gamma \setminus \mathcal D$.
\end{lemma}
\begin{proof}
Let us take an endpoint $x_*$ of a good interval $\mathcal I$ and prove that $x_* \in \mathcal D$. 
Suppose $x_*$ is not an endpoint of $\Gamma$, otherwise $x_* \in \mathcal D$ by criterion~1 from the above definition of $\mathcal D$.
Consider first the case when $\phi_{(0, \infty)}(x_*)$ intersects $\Gamma$, and let $y_*$ be the first point of intersection. Then $y_*$ is a boundary point of $\Gamma$, otherwise the trajectory segment connecting $x_*$ with $y_*$ would be a $\Gamma$-arc, and $x_*$ would be a good point. Then $x_* \in \mathcal D$ by criterion 2. 
Now only the case when $\phi_{(0, \infty)}(x_*)$ does not intersect $\Gamma$ remains.  By Lemma~\ref{l:good-intervals} the $\Gamma$-arcs $A(x)$, $x\in\mathcal I$, approach 
a one-sided limit $A_*$ as $x$ approaches $x_*$. The limit $A_*$ is a separatrix path or a trajectory segment that begins and ends on $\Gamma$.
As $\phi_{[0, \infty)}(x_*)$ does not intersect $\Gamma$, only the case of a separatrix path is possible. Then $x_*$ lies on a stable separatrix, so $x_* \in \mathcal D$ by the third criterion.

We have proved that both endpoints of any good interval must belong to the finite set $\mathcal D$. Because $(\phi_t)$ is continuous, the set of good points is dense in $\Gamma$ by the Poincar{\'e} recurrence theorem. 
Therefore, it contains $\Gamma \setminus \mathcal D$. To prove that the set of good points equals $\Gamma \setminus \mathcal D$, we note that for any $x \in \mathcal D$ there is clearly no $\Gamma$-arc starting at $x$, so $x$ is not good.
\end{proof}

\noindent We are now ready to prove that the first return maps are IETs.
\begin{proof}[Proof of Proposition~\ref{p:first-return-IET}]
As we have established above, $[a, b] \setminus \mathcal D$ is a finite union of good intervals. Note that for any good $x$, $P(x)$ is defined and there is a $\Gamma$-arc $A$ connecting $x$ with $P(x)$. Since $(\phi_t)$ is continuous and locally Hamiltonian, we see that the image of each good interval under $P$ is an interval of the same length. Furthermore, since $P$ is the Poincar{\'e} first return map, the images under $P$ of the good intervals are pairwise disjoint. Letting $T$ be the right-continuous extension of $P$, we obtain that $T$ is a right-continuous IET with the property that each of its discontinuities belongs to $\mathcal D$.
Finally, by Lemma \ref{l:good-intervals} for every $x\in[a,b)\cap\mathcal D$ at which $P$ is not defined there is a separatrix path connecting $x$ with $T(x)$ and, in the case that $x\in(a,b)$, there is a separatrix path connecting $x$ with $\lim_{y\rightarrow x^{-}}Ty$.
\end{proof}

\begin{proof}[Proof of Proposition~\ref{p:flows-Keane}]
Let $T:[a,b)\rightarrow[a,b)$ be the IET from Proposition~\ref{p:first-return-IET}.
If $x \in [a, b) \setminus \mathcal D$, then $T(x)$ is the point of first return of the forward semi-trajectory of $x$ to $\Gamma$, and thus $x$ and $T(x)$ are connected by a trajectory of $(\phi_t)$.
Take any $x_0 \in \mathcal D \cap [a, b)$, then $T(x_0)=\lim_{x \to x_0^+} T(x)$, and by Lemma~\ref{l:good-intervals} there is a trajectory or a separatrix path that connects $x_0$ with $T(x_0)$ (and if it is a trajectory, it passes through an endpoint of $\Gamma$). 
Iterating this shows that any two points on the same $T$-orbit are connected by a trajectory or by a separatrix path.
Thus, because $a$ does not lie on a stable separatrix or a periodic trajectory of $(\phi_t)$, we see that $T\ne\text{Id}$ and, so, $T$ has at least one discontinuity.

We are now ready to check Keane's condition. 
It is enough to prove that for any two discontinuities $x,y$ (which may coincide) of $T$ there is no positive $n\in\N$ with $T^nx=y$.
Notice that by the conditions imposed on $\Gamma$, there is at most one $x_0\in \mathcal D\cap (a,b)$ whose forward $(\phi_t)$-orbit goes through both endpoints of $\Gamma$. If such an $x_0$ exists, our assumptions about $\Gamma$ also imply that $x_0$ does not lie on an stable or unstable separatrix. Thus, in this case, for any two discontinuities $x,y$ of $T$ there is no $n\in\N$ with $T^nx=y$. (If such an $n$ existed we would either get that a separatrix connection intersects $\Gamma$ or that the $(\phi_t)$-orbit going through the endpoints of $\Gamma$ is either periodic or a separatrix.)
If such an $x_0$ does not exist, we again obtain that for any discontinuities $x,y$ of $T$ there is no $n\in\N$  for which $T^nx=y$ (the existence of such an $n$ implies that a separatrix connection intersects $\Gamma$). 
\end{proof}

\subsection{Sketch of the proof of Theorem~\ref{t:decomposition}}
Although Theorem~\ref{t:decomposition} is classical (see the book by Nikolaev and Zhuzhoma~\cite[Theorem 3.1.7]{NZ} and cited there works by Katok~\cite{Kat73} and Zorich~\cite{Zor94}, see also Maier's theorem~\cite{Mai}), the only proof we are aware of is a (sketch of) proof written in~\cite{NZ}, and following it involves going back to papers published in the first half of the 20th century. For reader's convenience, we provide a sketch of a different proof, which is based on the propositions established above.
\begin{enumerate}
\item Cut $\mm$ by all polycycles; the resulting domains will be the domains $\mm_i$. Clearly, their interiors are non-empty and disjoint. Unless there is only one domain $\mm_1 = \mm$, each domain claims one side of any separatrix connection bounding it, so there is a finite number of $\mm_i$.

\item Take $\mm_i$ that has a periodic orbit inside. Moving the periodic orbit by the orthogonal flow $(\psi_s)$ continues this periodic orbit in each direction.
By Proposition~\ref{p:psi-s-star}, it either degenerates into a polycycle or a single fixed point, or it can move indefinitely. 
If the periodic orbit degenerates into a polycycle for both directions, this shows that $\mm_i$ is an annulus filled by periodic orbits.
If it degenerates into a polycycle for one direction and into a point for another, we get a disc with a fixed  point inside.
If it degenerates into a point in both directions, we get the equatorial flow on the sphere.
If it continues indefinitely in any of the directions, it must return to the original orbit by compactness\footnote{Take any point $p$ on the periodic orbit $\gamma$ and suppose it has infinite positive $\psi_s$-trajectory. Let $\omega(p)$ be the omega-limit set of this trajectory. By compactness, $\omega(p)$ is non-empty and connected. If $\omega(p) \subset \mm_*$, by connectedness it must be just one point, and, so, the forward $\psi_s$-trajectory of $p$ cannot be infinite due to $\frac{dH}{ds}=1$. This shows that there is $q \in \omega(p) \setminus \mm_*$.
We now see by considering a small enough $\phi_t$-flowbox neighborhood of $q$ that there are $s_1,s_2>0$ for which $\psi_{s_1}(p),\psi_{s_2}(p),q$ have the same (periodic) $\phi_t$-trajectory. Thus, $\psi_{|s_2-s_1|}(\gamma)=\gamma$.}, so the whole $\mm$ is filled by periodic orbits. Then by Poincar\'e-Hopf index theorem $\mm$ is the torus.
In all these cases, we see that the closure of the continuous family of periodic orbits fills the whole $\mm_i$.
\item Let us now take $\mm_i$ that has no periodic orbits inside it and prove that it is a quasiminimal domain. Take any curve transversal to the flow $\tilde \Gamma\subseteq \mm_i$ and suppose that it contains its endpoints and is disjoint from every fixed point and separatrix connection. By Proposition \ref{p:first-return-IET}, there is an $x_0\in\tilde \Gamma$ for which $\phi_\R(x_0)$ is not a separatrix and for which there is a least $\tau>0$  with $\phi_\tau(x_0)\in\tilde\Gamma$. Let $\Gamma:=[a,b]$ be the segment of $\tilde \Gamma$ joining $x_0$ with $\phi_\tau(x_0)$ and let $T:[a,b)\rightarrow [a,b)$ and $\mathcal D$ be as defined in Proposition \ref{p:first-return-IET}. Let $a=a_0<\cdots<a_N=b$ be an enumeration of the members of $\mathcal D$. Note that $N>0$ and set 
\begin{equation}\label{e.DefnLambda}
\Lambda:=\overline{\bigcup_{x\in [a,b)}\phi_\R(x)} =\bigcup_{j=0}^{N-1}\overline{\{\phi_\R(x)\,|\,x\in (a_j,a_{j+1})\}}.
\end{equation}
Thus, by Lemma \ref{l:good-intervals} (and our assumptions on $\Gamma$), for any $x\in \Lambda$, if $\phi_\R(x)$ fails to be either a stable or an unstable separatrix, then $\phi_\R(x)$ intersects $[a,b)$.

We now check that if $\phi_\R(x)\subseteq \Lambda$ is not an stable separatrix, then $\phi_{[0,\infty)}(x)$ is dense in $\Lambda$. By the previous argument we see that $\phi_\R(x)$ intersects $\Gamma$. So, because Proposition~\ref{p:flows-Keane} implies that $T$ satisfies Keane's condition, we obtain that $\phi_{[0,\infty)}(x)$ is dense in $\Lambda$. Since the fact that $T$ satisfies Keane's condition implies that $T^{-1}$ also does, a similar argument shows that if $\phi_\R(x)\subseteq\Lambda$ is not an unstable separatrix, then $\phi_{(-\infty,0]}(x)$ is dense in $\Lambda$.

To complete the proof, it remains to show that $\Lambda=\mm_i$. First note that because  $\phi_\R(x_0)$ is dense in $\Lambda$, $\Lambda$ is connected. Also note that by \eqref{e.DefnLambda}, the boundary of $\Lambda$ is formed by fixed points and separatrix connections. Thus, because fixed points and separatrix connections do not separate $\mm_i$, we obtain the desired result.
\qed
\end{enumerate}

\subsection{Proof of Theorem~\ref{thm:AnalyticFlowRepresentation}}
 As was shown in the proof of Theorem~\ref{t:decomposition},
  Proposition~\ref{p:first-return-IET} and Proposition~\ref{p:flows-Keane} already give that $(\phi_t)|_{\mm_i}$ is isomorphic to a special flow over a minimal IET  whose roof function $f$ is defined as the time of first return of the forward $\phi_t$-trajectory of a point $x\in [a,b)$ to $[a,b)$. However, we must ensure  that the properties of the roof function $f$  are satisfied. This will be achieved by a particular choice of the transversal, namely, both its endpoints will lie on the same unstable separatrix of some saddle.
  
  \begin{proof}[Proof of Theorem~\ref{thm:AnalyticFlowRepresentation}] To proceed with this plan, we first need to find a point $q \in \intt \mm_i$ that lies on some unstable separatrix which is not a separatrix connection. Take an arbitrary transversal $\Gamma_2\subset \intt \mm_i$ disjoint from all separatrix connections of $(\phi_t)$ and consider an aperiodic $(\phi_t)$-trajectory $\gamma$ in $ \intt \mm_i$ which is not a separatrix. Because $\gamma$ is dense in $\mm_i$, it intersects $\Gamma_2$ in at least two points. Let $\Gamma_1 \subset \Gamma_2$ be the curve segment joining some two points of consecutive intersections of $\gamma$ with $\Gamma_2$. Notice that $\Gamma_1$ satisfies  the assumptions of Proposition~\ref{p:flows-Keane}.
  
  Consider the \emph{inverse time} first return map $P_{inv}: \Gamma_1 \mapsto \Gamma_1$. We claim that one can find a sequence of points $r_n\in\Gamma_1$ for which the $\phi_t$-trajectory joining $r_n$ with $P_{inv}(r_n)$ becomes arbitrarily long. Indeed, since $\mm$ has at least one $(\phi_t)$ fixed point, $\mm_i$ also has at least one $(\phi_t)$-fixed point (if $\mm_i \ne \mm$ there must be a fixed point on its boundary, and if $\mm_i = \mm$ it inherits the fixed point). Because $\mm_i$  is a quasi-minimal domain, points  in $\mm_i$  with  both forward and backward semi-trajectories being  arbitrarily close to this fixed point can be found. Both forward and backward semi-trajectories of any such point intersect $\Gamma_1$  in a dense subset. Thus, one can find a $(\phi_t)$-trajectory starting and ending at $\Gamma_1$ of arbitrarily large length. 
  
  By Proposition~\ref{p:first-return-IET} (applied in inverse time), any discontinuity point $x \in \Gamma_1$ of $P_{inv}$ either has the backward semi-trajectory passing through an endpoint of $\Gamma_1$, or this backward semi-trajectory is a piece of unstable separatrix of some saddle.
  As the return time is unbounded, there is a discontinuity point of separatrix type, which we take as $q$. By our assumptions on $\Gamma_1$, we have that $q$ is not an endpoint of $\Gamma_1$ (no separatrix goes through the endpoints of $\Gamma_1$) and $\phi_\R(q)$ is not a stable separatrix (no separatrix connection intersects $\Gamma_1$). 

  Now, take a  subinterval $\Gamma_0$ of $\Gamma_1$  that has $q$  as its right endpoint and is so short that it avoids separatrix connections, lies in $\intt \mm_i$, and the backward semi-trajectory of $q$ (which is an unstable separatrix) does not intersect $\Gamma_0$. As the forward semi-trajectory of $q$ is not a stable separatrix, it is dense in $\mm_i$. Take the point $q'$ where it first returns to $\Gamma_0$ and let $\Gamma \subset \Gamma_0$ be the segment connecting $q$ and $q'$. Then $\Gamma \cong [a, b]$ satisfies the assumptions of Proposition~\ref{p:flows-Keane}. 
  
 Let $T$ and $\mathcal D$ be as defined in Proposition~\ref{p:first-return-IET} and let $f:(a,b)\setminus\mathcal D\rightarrow (0,\infty)$ be the function defined by the property that $f(x)=\tau$ where $\tau$ is the least positive real number with $\phi_\tau(x)=Tx$ (In other words, $f$ is the time of first return of the $(\phi_t)$-forward trajectory of $x$ to $\Gamma$). Notice that by our choice of the endpoints of $\Gamma$, for every $x\in (a,b)\cap \mathcal D$, the forward $(\phi_t)$-trajectory of $x$ ends in a saddle before returning to $\Gamma$. Noting that by Proposition~\ref{p:flows-Keane} $T$ is minimal and that the forward $(\phi_t)$-trajectory of the endpoint $q$ (which goes through the other endpoint $q'$) is dense in $\mm_i$, it is easy to check that all the conditions in the statement of Theorem~\ref{thm:AnalyticFlowRepresentation} are satisfied. We leave the details to the reader.
  \end{proof}

\section{Analytic singularities of area-preserving flows} \label{s:analytic-singularities}
\subsection{Results}
In this section, we describe how an analytic Hamiltonian flow behaves near an isolated saddle (i.e., near an isolated fixed point which is not a local extremum of the Hamiltonian). Clearly, this applies to locally Hamiltonian flows as well. A particularly important feature for us is that the time it takes for a trajectory to pass a neighborhood of the critical point depends "nicely" on the difference $h$ between the values of the Hamiltonian at the trajectory and at the fixed point. Namely, it is given by a converging series involving rational powers of $h$, possibly multiplied by $\ln h$, see Definition~\ref{d:lp} below. This will be later used to study the special flow provided by Theorem~\ref{thm:AnalyticFlowRepresentation}. 

Let us now proceed with precise statements. 
Suppose $H(x, y)$ is an analytic function such that the origin $(0, 0)$ is an isolated critical point, which is not a local extremum. 
Subtracting a constant, we can assume $H(0, 0) = 0$.
Let $B_R$ denote the closed ball with radius $R$ centered at the origin $(0, 0)$ and let $\partial B_{R}$ be its boundary.
The critical level set of an analytic function near its critical point is described by the classical Puiseux Theorem.
However, for our purposes, the following simplified description is enough. 
\begin{lemma} \label{l:separatrices}
If $\rho$ is small enough, the zero level set $\{ (x, y) \in B_\rho : H(x, y) = 0 \}$ is a union of finitely many disjoint (except at the origin) analytic curves connecting the origin with $\partial B_{\rho}$, and they intersect $\partial B_{\rho}$ transversally. 
We will call these curves \emph{separatrices}.
\end{lemma}
\noindent This lemma is proved below in Section~\ref{s:resolution}.
Let $(\phi_t)$ denote the Hamiltonian flow~\eqref{e:hamilton-equations} given by $H$, the origin is an isolated fixed point of this flow. 
In a small neighborhood of the origin, the curves $H=0$ coincide with the (local) separatrices of the origin for the Hamiltonian flow $(\phi_t)$, thus justifying how we reuse the term "separatrix" that was already introduced in Section~\ref{s:flows1}.
Note that by Lemma~\ref{l:separatrices} an analytic fixed point of a locally Hamiltonian flow has finitely many separatrices. Therefore, \emph{the assumptions used in Sections~\ref{s:flows1} and~\ref{s:flows2} are satisfied by all analytic flows with finitely many fixed points.} In fact, instead of analyticity on the whole $\mm$, analyticity in neighborhoods of all saddle fixed points is sufficient.

Let us now describe the time it takes for the trajectories of $(\phi_t)$ to pass through a small neighborhood of the origin, which, as we will see, can be given by an asymptotic series of the following type:  
\begin{definition} \label{d:lp}
For a positive integer $n$ and an integer $k_0$, we will call a \emph{\lp series} a series of the form
\begin{equation} \label{e:lp}
  \sum \limits_{k=k_0}^\infty \left(a_k+b_k (\ln h)\right) h^{k/n},
  \qquad a_k, b_k \in \mathbb R.
\end{equation}    
\end{definition}
\noindent If all $b_k$ are zeroes, one gets the classical notion of Puiseux series, which is why we call such series "log-Puiseux". One can interpret a \lp series as a sum of two Puiseux series multiplied by $1$ and $\ln h$.
We will consider such series for $h \in (0, h_0)$ where $h_0$ is small enough.

Assume $\rho$ is small enough so that the conclusion of Lemma~\ref{l:separatrices} holds.
A set  $S_\rho \subset B_\rho$ is called a \emph{separatrix sector} if it is a connected component of $B_\rho \cap \{H \ne 0 \}$.

\begin{theorem} \label{t:time}
Suppose $\rho$ is small enough. Then, for any separatrix sector $S_\rho$ there are $h_0 > 0$ and a \lp series $P(h)$ converging on $(0, h_0)$ such that for each $0 < h < h_0$
\begin{enumerate}
    \item The subset of $S_\rho$ given by $|H(x, y)|=h$ is a trajectory of $(\phi_t)$ connecting two points on $\partial B_\rho$.  
    \item The time this trajectory spends in $S_\rho$ is equal to $P(h)$.
\end{enumerate}
\end{theorem}
\noindent This theorem is proved below in Section~\ref{s:resolution}. 

Finally, we state a corollary describing how the roof function $f$ in the special flow representation (Theorem~\ref{thm:AnalyticFlowRepresentation}) behaves near its singularities if the flow is analytic.
\begin{corollary} \label{c:floor-function}
Suppose that the assumptions of Theorem~\ref{thm:AnalyticFlowRepresentation} hold, and, additionally, $H$ is analytic in some neighborhood of each saddle of $(\phi_t)$. 
Let $f(x)$ be the roof function defined there and let $x_*$ be any of its infinite discontinuities. 
Fix a sufficiently small one-sided neighborhood $U$ of $x_*$. Then $f$ restricted to $U$ can be written as 
\[
f(x) = f_{sing}(x) + f_{reg}(x), \qquad x \in U,
\]
where $f_{sing}$ is the sum of a convergent Puiseux series w.r.t $h=|x-x_*|$, and $f_{reg}$ is $C^\infty$ up to $x_*$ (in the sense that $f_{reg}$ admits a $C^\infty$ extension to a two-sided neighborhood of $x_*$)\footnote{Note that $f_{sing}$ and $f_{reg}$ may be different for the left and right-sided neighborhoods of $x_*$.}.
\end{corollary}
\noindent This corollary is proved below in Section~\ref{s:floor-function}. It implies (as we show below in Lemma~\ref{l:return-time-properties}) that $f$ is of at least logarithmic growth~\eqref{e:log-growth-intro}, and thus Theorem~A follows from Theorem~A$^\prime$ and Theorem~\ref{thm:AnalyticFlowRepresentation}.  

\subsection{Resolution of singularities and its corollaries} \label{s:resolution}
In this subsection, we use the following classical (and highly nontrivial) resolution of singularities result to prove Theorem~\ref{t:time}.  
\begin{theorem}[Hironaka~\cite{Hir}, reproduced from~\cite{Ati}] \label{t:resolution}
Let $H(x, y)$ be a real analytic function that is defined in a neighborhood of $(0, 0)$, has a critical point $H(0, 0) = 0$, and is not identically zero. Then there exists an open set $U \ni (0, 0)$, a real analytic manifold $\tilde U$, and a proper analytic map $\pi: \tilde U \mapsto U$ such that 
\begin{enumerate}
    \item $\pi: \tilde U \setminus \tilde A \mapsto U \setminus A$ is an isomorphism, where $A = H^{-1}(0)$ and $\tilde A = \pi^{-1}(A)$,

    \item For each point $\tilde p \in \tilde U$ there are local coordinates $(x_1, y_1)$ centered at $\tilde p$ so that locally near $\tilde p$ we have
    \begin{equation} \label{e:x1-y1}
    H \circ \pi = c(x_1, y_1) x_1^n y_1^m,        
    \end{equation}
    where $c(x_1, y_1)$ is an analytic function with $c(x, y) \ne 0$ on its domain, and $n, m \ge 0$ are integers.
\end{enumerate}    
\end{theorem}
Let us start by introducing some common notation that will be used in the proofs below.
Let $B_r$ denote the closed ball with radius $r$ centered at the origin, and suppose $B_r \subset U$. Set $A_r = A \cap B_r$ and $\tilde A_r = \pi^{-1}(A_r)$. 
As $\pi$ is proper, $\tilde A_r$ is compact. 
Cover each point $p \in \tilde A_r$ by an open neighborhood $W \ni p$ such that there is a coordinate chart $(x_1, y_1)$ from Theorem~\ref{t:resolution} centered at $p$ and defined on $W$.
Then select a finite subcover $\{W_i\}$.
Using these local coordinates $(x_1, y_1)$, we see that $\tilde A_r$ is a union of finitely many pieces of analytic curves (locally given by $x_1 = 0$ or $y_1=0$) and points ($x_1 = y_1 = 0$) where those curves intersect. 

First, let us derive Lemma~\ref{l:separatrices} from Theorem~\ref{t:resolution}.   
\begin{proof}[Proof of Lemma~\ref{l:separatrices}]
 Set $D(x, y) = x^2 + y^2$ to be the squared distance to the origin and put $\tilde D(x_1, y_1) = D \circ \pi$. 
 Consider a curve piece (open segment) $\alpha \subset \tilde A_r$ covered by a single coordinate chart and given by either $x_1 = 0$ or $y_1 = 0$. Assume WLOG that $\alpha$ is given by $y_1 = 0$, then $x_1$ parametrizes $\alpha$, and we can identify $\alpha$ with an interval. Then $\tilde D(x_1)$ is an analytic non-negative function on this interval.
 One possibility is that $\tilde D$ is identically zero, then 
 $\pi(\alpha)$ is the origin. If this is not the case, analyticity implies that $\tilde D$ has only finitely many zeroes; suppose they occur at $x_1=a_1, \dots , a_k$.
 As $\tilde D \ge 0$, analyticity implies that for each zero there is $\delta_i > 0$ such that $\frac{d\tilde D}{dx_1} > 0$ on $(a_i, a_i + \delta_i]$ and $\frac{d\tilde D}{dx_1} < 0$ on $[a_i-\delta_i, a_i)$. 
 Put $\alpha_\rho = \alpha \cap \pi^{-1}(B_\rho)$; it is the subset of $\alpha$ given by $\tilde D \le \rho^2$.
 Decreasing $\rho$ is needed, we have 
 $\alpha_\rho \subset \left( \cup_i [a_i - \delta_i, a_i + \delta_i]\right)$. 
 Then, for each $i$ the set $\alpha_\rho \cap [a_i - \delta_i, a_i + \delta_i]$ is an interval $[c_i, d_i] \ni a_i$, and $\tilde D(c_i) = \tilde D(d_i) = \rho^2$. Hence, $\pi([c_i, a_i])$ is an analytic curve connecting the origin with $\partial B_\rho$. The signs of $\frac{d\tilde D'}{dx_1}$ established above show that its intersection with $\partial B_\rho$ is transversal.
 The same also holds for $\pi([a_i, d_i])$.
 Finally, we note that different curves $H=0$ can only intersect at a critical point of $H$, and, when $\rho$ is small enough, the origin is the only such point in $B_\rho$. 
\end{proof}

Let $S_\rho \subset B_\rho$ be a separatrix sector. Denote by 
\begin{equation} \label{e:sublevel}
S_\rho(h) = \{(x, y) \in S_\rho : |H(x, y)| < h\}     
\end{equation}
the sublevel set of $H$ in the separatrix sector.
\begin{lemma} \label{l:sector}
    Suppose $\rho$ is such that the origin is the only critical point in $B_\rho$ and the conclusion of Lemma~\ref{l:separatrices} holds in $B_\rho$. Then for any small enough $h$ the set $S_\rho(h)$ is connected and bounded by the following five curves:
    \begin{itemize}
        \item a trajectory of $(\phi_t)$ connecting two points on $\partial B_\rho$ such that $|H|=h$ on it, this trajectory is transversal to $\partial B_\rho$;
        \item two separatrices bounding the sector $S_\rho$;
        \item two arcs of the boundary circle $\partial B_\rho$ connecting the four points where the curves above intersect the boundary circle.  
    \end{itemize}  
\end{lemma}
\begin{proof}
    Let $C_h$ be the subset of $S_\rho$ where $|H(x, y)|=h$. Consider a connected component $C_{h, i}$ of this set. As the origin is the only critical point of $H$ in $B_\rho$, this connected component is a simple curve, which is a trajectory of $(\phi_t)$. It cannot be a closed curve, as then there would be a local extremum of $H$ inside it. Let $\Sigma_\rho$ be the arc of the boundary circle $\partial B_\rho$ that bounds the sector $S_\rho$.
    We see that $C_{h, i}$ intersects $\Sigma_\rho$ at two points. The value of $|H|$ is monotone on $\Sigma_\rho$ near the boundary points of $\Sigma_\rho$, as separatrices are transversal to the boundary circle. Take small open neighborhoods of the boundary points where we have the monotonicity, and let $h_0>0$ be the minimum of $|H|$ on $\Sigma_\rho$ outside these neighborhoods. Then any $h<h_0$ gives exactly two points of $\Sigma_\rho$ where this value of $|H|$ is achieved, so these are the points where $C_{h, i}$ intersects $\Sigma_\rho$. Hence, $C_h$ is a single level curve intersecting $\Sigma_\rho$ at two points. 
    
    By Lemma~\ref{l:separatrices}, we know that separatrices are transversal to $\partial B_\rho$. As $C_h$ intersects the boundary circle near its intersections with separatrices, $C_h$ is transversal to $\partial B_\rho$ as well. It divides $S_\rho$ into the parts where $|H|<h$, which is $S_\rho(h)$, and  where $|H|>h$.
    This gives the description of the boundary of $S_\rho(h)$ claimed in the lemma.
\end{proof}

We now state a key theorem that will imply Theorem~\ref{t:time}. This theorem, as well as its analogues in arbitrary dimension and further generalizations, are known: see~\cite{Jean, AGV, Loes, Green} and references therein. However, the proof is simple modulo Theorem~\ref{t:resolution}, so we give a complete proof for the reader's convenience.  
\begin{theorem} \label{t:area}
Let $H(x, y)$ be a real-analytic function such that $H(0, 0) = 0$ is a critical point, which is not a local extremum. Fix a small enough $\rho>0$ and a separatrix sector $S_\rho$. 
Then there exist $h_0>0$ and a \lp series $P(h)$, which converges on $(0, h_0)$, such that for any $h \in (0, h_0)$ the area of the sublevel set $S_\rho(h)$ defined by~\eqref{e:sublevel} equals $P(h)$.
\end{theorem}
\begin{proof}
As noted after Theorem~\ref{t:resolution}, the set $\tilde A_r$ is compact, and is a union of finitely many pieces (closed segments) of analytic curves with different pieces intersecting at a finite set of points.
Cut each of those pieces into finitely many subpieces such that each subpiece is covered by some local coordinates $(x_1, y_1)$ from Theorem~\ref{t:resolution}. At each of the cutting points, select an analytic transversal to $\tilde A_r$. Then, those transversals cut any small enough neighborhood of $\tilde A_r$ into pieces, each of which is covered by  some local coordinates $(x_1, y_1)$. Call these transversals \emph{cutting transversals}.

Let $\tilde S_\rho = \pi^{-1}(S_\rho)$ to be the lift of the separatrix sector $S_\rho$. Put $\tilde S_\rho(h) = \pi^{-1}(S_\rho(h))$, then $\tilde S_\rho(h) \subset \tilde S_\rho$ is given by the inequality $|H \circ \pi| < h$. 
Let $G(h)$ denote the area of $S_\rho(h)$.
Denoting by $J_\pi$ the Jacobian of $\pi$, we get
\begin{equation} \label{e:int-resolved}
G(h) = \int_{\tilde S_\rho(h)} \pi^*(dx \wedge dy),    
\end{equation}
where in the local coordinates $\pi^*(dx \wedge dy) = J_\pi dx_1 \wedge dy_1$. 
Using the local coordinates, we see that, taking $h$ to be small enough, one can fit $\tilde S_r(h)$ inside any neighborhood of $\tilde A_r$. Pick $\delta > 0$ such that $\tilde S_r(\delta)$ is divided by the cutting transversals into pieces, each of which is covered by a single local coordinate chart $(x_1, y_1)$ centered at a point of $\tilde A_r$. Additionally, cut each of those pieces by the coordinate axes $x_1=0$ and $y_1=0$ to obtain a decomposition of $\tilde S_r(\delta)$ into \emph{elementary pieces}, which we denote by $V_i$. For $h < \delta$, put $\tilde S_{r, i}(h) = V_i \cap \tilde S_r(h)$. Then the integral giving $G(h)$ is decomposed into the sum of the integrals over $\tilde S_{r, i}(h)$.
Each of those writes in the local coordinates as 
\[
G_i(h) = \int_{\tilde S_{r, i}(h)} J_\pi(x_1, y_1) dx_1 dy_1.
\]
The domain of integration $\tilde S_{r, i}(h)$ is the subset of $V_i$ given by $|H \circ \pi| < h$, which is written in the local coordinates as 
\begin{equation} \label{e:ineq-h}
|c_1(x_1, y_1)x_1^n y_1^m| < h.    
\end{equation}
Interchanging $x_1$ and $y_1$ if necessary, we can assume $n>0$. Perform the invertible analytic coordinate change $\psi: (x_1, y_1) \mapsto (x_2, y_2)$, where $x_2 = |c_1(x_1, y_1)|^{1/n} x_1$ and $y_2=y_1$. The inequality~\eqref{e:ineq-h} becomes simply $|x_2^n y_2^m| < h$, so we have
\[
G_i(h) = \int_{\psi(V_i) \cap \{|x_2^n y_2^m| < h\}} \frac{J_\pi \circ \psi^{-1}}{J_\psi} \; dx_2 dy_2,
\]
where $J_\psi \ne 0$ is the Jacobian of $\psi$. The integrand is an analytic function.
The boundary of $\psi(V_i)$ is given by coordinate axes, analytic transversals to them (at most one transversal to each coordinate axis), and the curve $|x_2^n y_2^m| = \delta$ (which is irrelevant for the domain of integration as we only consider $h<\delta$).
We now state a lemma that will imply that $G_i(h)$ is given by a \lp series.
\begin{lemma} \label{l:integral-computation}
    Let $f(x, y)$ be an analytic function and $X(y)$ and $Y(x)$ be analytic functions that are positive in a neighborhood of zero. Let $n, m \ge 0$ be integers with $n > 0$. Let $V$ be the domain given by $0 < x < X(y)$ and $0 < y < Y(x)$. Then for a small enough $h_0 > 0$ the integral
    \[
    I(h) = \int_{V \cap \{x^ny^m < h\}} f(x, y) \; dx dy
    \]
    is given for $h \in (0, h_0)$ by a convergent \lp series.
\end{lemma}
\noindent This lemma is proven in Appendix~\ref{s:integrals-puiseux} by decomposing the integrand into a Taylor series and directly computing the integrals of the monomials. 
It implies that each $G_i(h)$ is a \lp series, and so is their sum $G(h)$. This completes the proof of Theorem~\ref{t:area}.
\end{proof}

\begin{lemma} \label{l:dA-dh}
Fix a small enough $\rho>0$. Let $A(h)$ be the area of $S_\rho(h)$, and $\tau(h)$ be the time the trajectory of $(\phi_t)$ with $|H(x, y)|=h$ spends in $S_\rho$.
Then for all small enough $h>0$ we have 
\[
\tau(h) = \frac{d}{dh} A(h).
\]
\end{lemma}
\begin{proof}
  Consider new coordinates on $S_\rho$: the value $\tilde h$ of $|H(x, y)|$ and the time $\tilde t$ it takes for a trajectory starting from the given point to reach the boundary circle $\partial B_\rho$.
  It is well-known and can be checked by a simple computation that the change of variables $(x, y) \mapsto (\tilde h, \tilde t)$ is area-preserving.
  So, $A(h)$ can be computed in the new coordinates, giving $A(h)=\int_0^h \tau(\tilde h)d \tilde h$. By the fundamental theorem of calculus, $\frac{dA}{dh} = \tau(h)$.
\end{proof}

\noindent We are now ready to prove Theorem~\ref{t:time}. 
\begin{proof}[Proof of Theorem~\ref{t:time}]
The first claim follows from Lemma~\ref{l:sector}.
To prove the second claim, we note that by Theorem~\ref{t:area} the area $A(h)$ is given by a convergent \lp series. Taking the derivative yields the \lp series for the time $\tau(h)$ by Lemma~\ref{l:dA-dh}.
\end{proof}

\subsection{Roof function near its discontinuities} \label{s:floor-function}
In this subsection we use Theorem~\ref{t:time} to derive Corollary~\ref{c:floor-function}. We then state and prove Lemma~\ref{l:return-time-properties}, which is needed to prove Theorem~A.
\begin{proof}[Proof of Corollary~\ref{c:floor-function}]
Let $\Gamma$ be the transversal chosen in the proof of Theorem~\ref{thm:AnalyticFlowRepresentation}.
Recall that the roof function $f(x)$ is the first return time to $\Gamma$, and the coordinate $x$ on it is  the value of a local Hamiltonian.

Take any infinite discontinuity $x_* \in \Gamma$ of $f(x)$ and select one of its one-sided neighborhood.
By Lemma~\ref{l:good-intervals} the corresponding one-sided limit when $x \to x_*$ of $\Gamma$-arcs starting at $x$ is a separatrix path (see the definition in Section~\ref{s:orthogonal}) starting at $x_*$ and ending at some point $x_*' \in \Gamma$.
Let $s_1, \dots, s_N$ be its saddles and let $\gamma_1, \dots, \gamma_{N-1}$ be the separatrix connections forming this path.
Let us also denote by $\gamma_0$ the separatrix piece connecting $x_*$ with $s_1$ and by $\gamma_{N}$ the separatrix piece connecting $s_N$ with $x_*'$.
Then $\gamma_{i-1}$ is a stable separatrix of $s_i$ and $\gamma_i$ is an unstable separatrix of $s_i$, where $i=1, \dots, N$.

For any saddle $s_i$, $i=1, \dots, N$, in this path, pick a small $\rho_i > 0$ such that the balls $B_{\rho_i}(s_i)$ are disjoint, in each of these balls there is an analytic local Hamiltonian $H_i$, we have the conclusion of Lemma~\ref{l:separatrices} in all these balls, and we have the conclusion of Theorem~\ref{t:time} for each separatrix sectors in each of these balls $B_{\rho_i}(s_i)$.
Denote by $\Lambda_i^s$ and $\Lambda_i^u$ small arcs of $\partial B_{\rho_i}(s_i)$ around its intersections with $\gamma_{i-1}$ and $\gamma_i$, respectively. 
Also, take $\Lambda_0^u,\Lambda_N^s = \Gamma$. 
Then, in a small enough one-sided neighborhood of $x_*$ the first return map to $\Gamma$ decomposes into the composition of the Poincar\'e maps from $\Lambda_i^u$ to $\Lambda_{i+1}^s$ and from $\Lambda_i^s$ to $\Lambda_i^u$. For the maps of the first type, the Poincar\'e map is defined at the point where the separatrix path intersects the preimage transversal, so the corresponding passing times are $C^\infty$. For the maps of the second kind, the passing time is a \lp series by Theorem~\ref{t:time}. Adding those times together gives Corollary~\ref{c:floor-function}.   
\end{proof}

\begin{lemma} \label{l:return-time-properties}
If $H$ is analytic near each saddle of $(\phi_t)$, the function $f(x)$ from Theorem~\ref{thm:AnalyticFlowRepresentation}  is of at least logarithmic growth, i.e., satisfies condition~\eqref{e:log-growth-intro}.
\end{lemma}
\begin{proof}
Let us first restrict ourselves to a small enough one-sided neighborhood $U$ of some singularity $x_*$ of $f(x)$ and prove that we have
\begin{equation} \label{e:log-local}
    f''(x) \ge \frac {C_1}{(x-x_*)^2}
\end{equation}
in $U$ for some $C_1=C_1(U) > 0$.
By Corollary~\ref{c:floor-function}, in $U$ the function $f(x)$ is a sum of a $C^\infty$ function and a function which is a sum of some converging \lp series.
Let $f_{lead}(x)$ be the function given by the leading term of this series only.
As $f(x)$ has an infinite singularity, the leading term can only be be one of the following:
\[
-c\ln(h), \; ch^{-r}, -c\ln (h)h^{-r},
\qquad c, r > 0.
\]
The function $f_{lead}(x)$ is obtained by plugging in $h=|x-x_*|$, and we see that it satisfies~\eqref{e:log-local} (with $f_{lead}''$ instead of $f''$) in each of the three cases above.
Making $U$ smaller if necessary, we see that the leading term dominates all the other terms of the series, as well as the smooth function, so $f''$ also satisfies~\eqref{e:log-local}, but with a smaller constant.

Clearly, one can select a uniform constant $C_1$ such that~\eqref{e:log-local} holds with this $C_1$ in sufficiently small right and left one-sided neighborhoods of each singularity $x_*$ of $f(x)$.
Now we can get Condition~\eqref{e:log-growth-intro} by picking $C$ there to be equal to the uniform $C_1$, and $B$ there to be positive and big enough so that~\eqref{e:log-growth-intro} is satisfied outside the selected neighborhoods of singularities as well (note that condition 4 in  Theorem~\ref{thm:AnalyticFlowRepresentation} plays a crucial role here).
\end{proof}

\section{Induced IETs with an "optimal" number of discontinuities} \label{s:induced-IET}
Let  $T:[0,1)\rightarrow [0,1)$  be a right-continuous IET and let $\Delta:=[a,b)\subseteq [0,1)$.  A point $x\in\Delta$ is called a \textit{$T$-cut on $\Delta$} if there exists an $i\in\{0,...,h_{\Delta,T}(x)-1\}$ such that $T^ix$ is a discontinuity of $T$. Our goal in this section is to prove the following result. We say that a right-continuous  IET $T:[0,1)\rightarrow [0,1)$ is aperiodic if for every $x\in [0,1)$ and every $n\in\N$, $T^nx\neq x$.
\begin{theorem}\label{thm:RegularDisc}
    Let $T:[0,1)\rightarrow [0,1)$ be an aperiodic, right-continuous  IET and let $\epsilon>0$. There exist an $M>1$ and points $y_1,...,y_{M-1}\in (0,1)$, $y_1<\cdots<y_{M-1}$, such that:
    \begin{enumerate}
    \item [(i)] Setting $y_0=0$ and $y_M=1$, we have
    $$
\max_{0\leq j<M}|y_{j+1}-y_j|\leq \epsilon.
    $$
    \item [(ii)] Let $j\in\{0,...,M-1\}$ and set $\Delta=[y_j,y_{j+1})$. Every discontinuity of $T_{\Delta}$ is a $T$-cut on $\Delta$. Furthermore, if $y$ is the discontinuity of $T_\Delta$ with $\lim_{x\rightarrow y^-}T_\Delta x=y_{j+1}$, then there is $r\in \{0,...,i-1\}$ with $T^r y$ a discontinuity of $T$, where $i=\lim_{x\rightarrow y^-}h_{\Delta,T}(x)$.  
    \end{enumerate}
\end{theorem}
As we explain in Subsection \ref{Sec:ProofOfNiceLattice} below, Theorem \ref{thm:RegularDisc} is a consequence of Proposition \ref{prop:TypesOfDiscontinuities} and the following consequence of Theorem 5.1.1 in \cite{CornfeldErgodicTheory}.
\begin{lemma}[Cf. Theorem 5.1.1 in \cite{CornfeldErgodicTheory}.]\label{lem:DensenesInAperiodicT} Let $T:[0,1)\rightarrow [0,1)$ be a right-continuous IET with at least one discontinuity. Let $0<d_1<\cdots<d_{N-1}<1$ be a list of the discontinuities of $T$. Then, $T$ is aperiodic if and only if
$$
\bigcup_{j=1}^{N-1}\{T^nd_j\,|\,n\in\N\}
$$
is dense in $[0,1)$.
\end{lemma}
For completeness,
we present the proof of Lemma \ref{lem:DensenesInAperiodicT} next in Subsection \ref{Sec:PrelimPRoofNiceLAttice}.
\subsection{The proof of Lemma \ref{lem:DensenesInAperiodicT}}\label{Sec:PrelimPRoofNiceLAttice}
We will derive Lemma \ref{lem:DensenesInAperiodicT} from the following dichotomy principle dealing with induced IETs.
\begin{lemma}\label{lem:InducedIETDichotomy}
Let $T:[0,1)\rightarrow [0,1)$ be a right-continuous IET and let $a,b\in [0,1)$ be such that $a<b$. Then, either (I) at least one $x\in [a,b)$ is a $T$-cut on $[a,b)$ or (II) $T_{[a,b)}x=x$ for every $x\in [a,b)$.
\end{lemma}
\begin{proof}
We will show that if no $x\in[a,b)$ is a $T$-cut on $[a,b)$,  then $T_{[a,b)}x=x$ for every $x\in[a,b)$.  Thus, assume that for every $x\in[a,b)$ and every $i\in\{0,...,h_{[a,b),T}(x)-1\}$, $T$ is continuous at $T^ix$ (and, so, $T$ is continuous at $T^ix$ for each $i\in \N\cup\{0\}$).\\
Let  $\mathcal I_1,...,\mathcal I_M$ be a list of the maximal subintervals of $[0,1)$ exchanged by $T$. Since $T$ is continuous on $[a,b)$, we can assume by re-indexing the $\mathcal I_j$s, if needed, that without loss of generality $[a,b)\subseteq \mathcal I_1$.\\
Let $\mathcal J$ be the maximal connected component of the set
$$\mathcal T:=\{T^ix\,|\,x\in[a,b)\text{ and }i\in\{0,...,h_{[a,b),T}(x)-1\}\}\cap\mathcal I_1$$
with the property that $\mathcal J\cap[a,b)\neq\emptyset$. Note that for every $y\in \mathcal J$, there is an $x\in [a,b)$ and an $i\in \{0,...,h_{[a,b),T}(x)-1\}$ with $T^ix=y$. Thus, we have that for every $i\in\N\cup\{0\}$, $T$ is continuous on $T^i\mathcal J$. Letting $i_0$ be the least $n\in\N$ for which $T^n\mathcal J\cap \mathcal J\neq \emptyset$, we see  that for every $i\in\{0,...,i_0-1\}$, $T^{i+1}\mathcal J$ is an interval and, so,  there is a $j_i\in\{1,...,M\}$ with $T^{i+1}\mathcal J\subseteq \mathcal I_{j_i}$ (otherwise, $T$ would be discontinuous on $T^{i+1}\mathcal J$ and, so, there would be an $x\in[a,b)$ which is a $T$-cut on $[a,b)$). In turn, this implies that for some $\sigma\in\R$, $T^{i_0}\mathcal J=\mathcal J+\sigma$. \\
Noting that $\mathcal J\cup T^{i_0}\mathcal J$ is an interval and that $\mathcal J\cup T^{i_0}\mathcal J \subseteq \mathcal T$, we conclude, by the maximality of $\mathcal J$, that $T^{i_0}\mathcal J=\mathcal J+\sigma\subseteq \mathcal J$. Because this can only happen when $\sigma=0$ and since for every $x\in[a,b)$, $h_{[a,b),T}(x)\geq i_0$, we conclude that $T_{[a,b)}x=x$ for every $x\in[a,b)$.
\end{proof}
\begin{proof}[Proof of Lemma \ref{lem:DensenesInAperiodicT}.]
    If $T$ is not aperiodic, then we can find an $a\in [0,1)$ and an $h\in\N$ for which $T^ia\neq a$ for $i\in\{1,...,h\}\setminus\{h\}$ and $T^ha=a$. It follows from the right continuity of $T$  that there is a $b\in (a,1)$ for which the sets
    $$ T^i[a,b),\,i\in\{0,...,h-1\},$$
    are pairwise disjoint intervals and $T^h[a,b)=[a,b)$. In turn, this implies that $T_{[a,b)}x=x$ for each $x\in [a,b)$ and, so, there is an $\epsilon>0$ for which
    $$
\bigcup_{j=1}^{N-1}\{T^nd_j\,|\,n\in\N\}\cap (a,a+\epsilon)=\emptyset.
    $$

    Suppose now that $T$ is aperiodic. In order to prove that $\bigcup_{j=1}^{N-1}\{T^nd_j\,|\,n\in\N\}$ is dense in $[0,1)$, we will show that for any $a,b\in [0,1)$ with $a<b$, there is a $j\in\{1,..,N-1\}$ and a $k\in\N$ with $T^kd_j\in[a,b)$.\\
    To do this, first note that because  $T$ is aperiodic, one cannot have $T_{[a,b)}x=x$ for each $x\in [a,b)$. So, by
    Lemma \ref{lem:InducedIETDichotomy}, there is an $x\in[a,b)$ which is a $T$-cut on $[a,b)$. In other words, one can find $x,y\in [a,b)$, a $j\in\{1,...,N-1\}$, and an $i\in\{0,...,h_{[a,b),T}(x)-1\}$  with $T^{h_{[a,b),T}(x)}x=y$ and $T^ix=d_j$. Thus, taking $k=h_{[a,b),T}(x)-i$, we get $T^kd_j=y$. We are done.
\end{proof}

\subsection{The proof of Theorem \ref{thm:RegularDisc}}\label{Sec:ProofOfNiceLattice}
\begin{proof}[Proof of Theorem \ref{thm:RegularDisc}.]
Because $T$ is aperiodic, it has at least one discontinuity. We will let
$$
0<d_1<\cdots<d_{N-1}<1
$$
denote the discontinuities of $T$ and set $d_0=0$ and $d_N=1$. Fix $\epsilon>0$. In order to prove Theorem \ref{thm:RegularDisc} we will first find $y_0,...,y_{M}$ and show that they satisfy item (i) in Theorem \ref{thm:RegularDisc} and then prove that $y_0,...,y_{M}$ also satisfy (ii) in Theorem \ref{thm:RegularDisc}.\\
\qedsymbol \textit{ Defining $y_0,...,y_{M}$ and checking that (i) holds.} By Lemma \ref{lem:DensenesInAperiodicT}, we have that $$\bigcup_{j=1}^{N-1}\{T^nd_j\,|\,n\in\N\}$$
is dense in $[0,1)$. Thus, there exists a $K\in\N$ such that if we let $y_0<y_1<\cdots<\cdots<y_M$ be an enumeration of
$$\bigcup_{j=1}^{N-1}\{T^n d_j\,|\,n\in\{1,...,K\}\}\cup\{0,1\},$$
we obtain that
$$\max_{0\leq j< M}|y_{j+1}-y_j|\leq\epsilon.$$
\qedsymbol \textit{ Checking that (ii) holds.} Let $j\in\{0,...,M-1\}$, set $\Delta=[y_j,y_{j+1})$, and let $y$ be a discontinuity of $T_{[y_j,y_{j+1})}$. By Proposition \ref{prop:TypesOfDiscontinuities}, all that we need to check is that (I) if $y$ is such that $T_\Delta y=y_j$, then $y$ is a $T$-cut on $\Delta$ and that  (II) if there exists an $i\in\{1,...,h_{\Delta,T}(y)-1\}$ with $\lim_{x\rightarrow y^-}T^ix=y_{j+1}$, then there is an $r\in \{0,...,i-1\}$ with $T^ry$ a discontinuity of $T$.\\
\qedsymbol \textit{ $T_\Delta y=y_j$.} Suppose first that $T_\Delta y=y_j$. Note that there exists a discontinuity $\alpha$ of $T$ and an $n\in\{1,...,K\}$ for which $T^n\alpha=y_j$ (in the case that $y_j=0$, one can take $\alpha$ to be the discontinuity $\beta$ of $T$ for which $T\beta=0$). Noting that
$$
\bigcup_{j=1}^{N-1}\{T^nd_j\,|\,n\in\{1,...,K\}\}\cap (y_j,y_{j+1})=\emptyset,
$$
we obtain that there exists an $i\in\{0,...,h_{\Delta,T}(y)-1\}$ with $T^iy\in\{d_1,...,d_{N-1}\}$. So, in particular, $y$ is a $T$-cut on $\Delta$.\\
\qedsymbol \textit{ There exists an $i\in\{1,...,h_{\Delta,T}(y)-1\}$ with $\lim_{x\rightarrow y^-}T^ix=y_{j+1}$.} Suppose now that  there exists an $i\in\{1,...,h_{\Delta,T}(y)-1\}$ with $\lim_{x\rightarrow y^-}T^ix=y_{j+1}$. It suffices to show  that there is an  $r\in\{0,...,i-1\}$ with $T^ry$ a discontinuity of $T$. If this was not the case, we would have that for each $r\in\{0,...,i-1\}$, $T$ is continuous at $T^ry$ and, so, that $T^i$ is continuous at $y$.
This yields that 
$$
y_{j+1}=\lim_{x\rightarrow y^-}T^ix=T^iy. 
$$
So, by arguing as above, we obtain that for some $r\in \{0,...,i-1\}$, $T^ry$ is a discontinuity of $T$, a contradiction. We are done.
\end{proof}
\section{Proof of Theorem A$^\prime$}
Let $\mathcal A=\{y_1,...,y_{D}\}$, $D\in\N$, be a subset of $(0,1)$ and let $d_1,...,d_{N-1}\in\mathcal A$, $N>1$, be such that $0<d_1<\cdots<d_{N-1}<1$. Throughout this section  we will let $T:[0,1)\rightarrow [0,1)$ denote a right-continuous IET whose list of discontinuities is given by $\{d_1,...,d_{N-1}\}$ and let $f:([0,1)\setminus \mathcal A)\rightarrow [0,\infty)$ be a member of $C^2([0,1)\setminus\mathcal A)$. We further assume that for every $j\in\{1,...,D\}$,  $f$ satisfies
$$\lim_{x\rightarrow y_j^+}f(x)=\lim_{x\rightarrow y_j^-}f(x)=\infty.$$ \\
We let $T^f$ denote the special flow over $T$ and under the roof function $f$.
\begin{definition}\label{defn:SuperLog+Regular}
Let $T$ and $f$ be as above. The function $f$ is said to be of  \textit{at least logarithmic growth} if there exist  $B,C>0$ such that for every $x\in (0,1)\setminus \mathcal A$, one has
\begin{equation}\label{eq:SuperLog}
f''(x)\geq \frac{C}{\min_{y\in \mathcal A}|x-y|^2}-B.
\end{equation}
\end{definition}
\begin{proof}[Proof of Theorem \hyperlink{TheoremA'}{A$^{\prime}$}]
Note first that if $T$ is non-ergodic and $A$ is  a non-trivial $T$-invariant subset of $[0,1)\setminus\mathcal A$, then the set
$$
\{T^f_tx\,|\, x\in A\text{ and }t\in \R\} 
$$
is a non-trivial $T^f$-invariant subset of 
$$\{T_t^fx\,|\, x\in [0,1)\setminus\mathcal A\text{ and }t\in [0,f(x))\}.$$
Thus, assume now that $T$ is ergodic and suppose for the sake of contradiction that there exists a measurable function $\psi:[0,1)\rightarrow  \mathbb S^1$ with the property that for some non-zero $s\in\R$,
\begin{equation}\label{Defn:eigenFunction}
\psi(Tx)=e^{2\pi i s f(x)}\psi(x)
\end{equation}
for each $x\in[0,1)\setminus \mathcal A$. By replacing $\psi$ with its conjugate $\overline \psi$, if needed, we can assume without loss of generality that $s=|s|$. 
To arrive to the desired contradiction,  we will find $\epsilon\in(0,1/2)$, $h\in\N$, 
and $x_1,x_2\in [0,1)$ for which the expression 
$$ |s|\left(\sum_{i=0}^{h-1}f(T^ix_2)-\sum_{i=0}^{h-1}f(T^ix_1)\right)$$
is, simultaneously, $\epsilon$-close and $\epsilon$-away from $0\mod 1$. On the one hand, we will deduce from \eqref{Defn:eigenFunction}  that there is an $a\in\Z$ for which 
\begin{equation}\label{eq:FlatRoof}
\left| |s|\left(\sum_{i=0}^{h-1}f(T^ix_2)-\sum_{i=0}^{h-1}f(T^ix_1)\right)-a\right|<\epsilon.
\end{equation}
On the other hand, the fact that $f$ is of at least logarithmic growth will allow us to show that  
\begin{equation}\label{eq:ConcaveRoof}
\inf_{a\in\Z}\left| |s|\left(\sum_{i=0}^{h-1}f(T^ix_2)-\sum_{i=0}^{h-1}f(T^ix_1)\right)-a\right|>\epsilon
\end{equation}
\subsection{Preliminary dynamical constructions}
\qedsymbol \textit{ A sequence of convenient Rohlin towers.}
Because the eigenfunction  $\psi$ is measurable and $T$ is ergodic (which, in turn, implies that $T$ is aperiodic), Theorem \ref{thm:RegularDisc}
implies the existence of a  sequence of subintervals $\mathcal I_\ell:=[A_\ell,B_\ell)$, $\ell\in\N$, of $[0,1)$ with $\lim_{\ell\rightarrow\infty}m(\mathcal I_\ell)=0$ and a sequence of constants $(\kappa_\ell)_{\ell\in\N}$ in $\mathbb S^1$ such that for each $\ell\in\N$, 
\begin{enumerate}
\item [($\mathcal I_{\ell}.1$)]Every discontinuity $x$ of $T_{\mathcal I_\ell}$ is a $T$-cut on $\mathcal I_\ell$. Furthermore, if $y\in [A_\ell,B_\ell)$ is the discontinuity of $T_{\mathcal I_\ell}$ with $\lim_{x\rightarrow y^-}T_{\mathcal I_\ell}x=B_\ell$, then there is an $r\in \{0,...,i-1\}$ with $T^ry$ a discontinuity of $T$, where $i=\lim_{x\rightarrow y^-} h_{\mathcal I_\ell,T}(x)$.
    \item [($\mathcal I_{\ell}.2$)] The set,
$$
E_\ell:=\{x\in\mathcal I_\ell\,|\,|\psi(x)-\kappa_\ell|\geq\frac{1}{\ell+1}\}
$$
satisfies 
$$
\frac{m(\mathcal I_\ell\cap E_\ell)}{m(\mathcal I_\ell)}<\frac{1}{(D+1)(\ell+1)}.
$$
\end{enumerate}
Note that because $T$ is  aperiodic, $T_{\mathcal I_\ell}$ has at least 
one discontinuity on $\mathcal I_\ell$. By  ($\mathcal I_\ell$.1), we can find  a sequence of subintervals $\mathcal J_\ell:=[a_\ell,b_\ell)\subseteq \mathcal I_\ell$, $\ell\in\N$,  and a sequence  $(h_\ell)_{\ell\in\mathbb N}$ in $\mathbb N$, with the following properties:
\begin{enumerate}
    \item [($\mathcal J_\ell.0$)] One has that 
    $$
\frac{b_\ell-a_\ell}{m(\mathcal I_\ell)}\geq \frac{1}{D+1}
    $$
    and for each $x\in [a_\ell,b_\ell)$, $h_\ell=h_{\mathcal I_\ell,T}(x)$.
    \item [($\mathcal J_\ell.1$)] The collections of subsets $T^i[a_\ell,b_\ell)$, $i\in \{0,...,h_\ell-1\}$, is a collection of pairwise disjoint intervals and $\lim_{\ell\rightarrow\infty}(b_\ell-a_\ell)=0$.
    \item [($\mathcal J_\ell.2$)] For each $x\in (a_\ell,b_\ell)$ and each $i\in\{0,...,h_\ell-1\}$, $T^ix\not\in\mathcal A$ and, so,  $f(T^ix)$ and  $f''(T^ix)$ are well-defined.
    \item [($\mathcal J_\ell.3$)] For some $j\in \{0,...,h_\ell-1\}$ and some $s\in \{1,...,D\}$, one either has $T^ja_\ell=y_s$ or $\lim_{x\rightarrow b_\ell^-}T^jx=y_s$.\footnote{ We remark that the second part of ($\mathcal I_\ell$.1) (and that of item (ii) in Theorem \ref{thm:RegularDisc}) play a fundamental role in the deduction of ($\mathcal J_\ell$.3). Indeed, these clarifying comments provide relevant information about
     the least $r\in\{0,...,h_{\mathcal I_\ell,T}(y)-1\}$,  $y$ a discontinuity of $T_{\mathcal I_\ell}$, for which $T^ry$ is a discontinuity of $T$.}
    \item [($\mathcal J_\ell.4$)] The set $T^{h_\ell}[a_\ell,b_\ell)$ is a subinterval of $\mathcal I_\ell$.
    \item [($\mathcal J_\ell.5)$]  
    We have
    $$\frac{m((a_\ell,b_\ell)\cap E_\ell)}{m((a_\ell,b_\ell))},\frac{m(T^{h_\ell}(a_\ell,b_\ell)\cap E_\ell)}{m(T^{h_\ell}(a_\ell,b_\ell))}<\frac{1}{\ell+1}.$$
\end{enumerate}
\qedsymbol \textit{ Picking $\epsilon>0$ and other relevant constants.}
For each $\ell\in\N$ and each $x\in (a_\ell,b_\ell)$, let
$$F_\ell(x):=|s|\sum_{i=0}^{h_\ell-1}f(T^ix)$$
 and let $C,B>0$ be as in Definition \ref{defn:SuperLog+Regular}. Invoking 
  condition $(\mathcal J_\ell.3)$ we obtain that for some $C_1,C_2,c>0$ one has that for every $\ell\in\N$ large enough, 
 \begin{equation}\label{eq:LowerBoundForBirkhoff.1}
F_\ell''(x)\geq \frac{C_1}{(b_\ell-a_\ell)^2}- |s|h_\ell B\geq \frac{C_1}{(b_\ell-a_\ell)^2}- \frac{C_2}{b_\ell-a_\ell}\geq \frac{c}{(b_\ell-a_\ell)^2}
 \end{equation}
for every $x\in (a_\ell,b_\ell)$.

Set $\epsilon=\min\{\frac{1}{4},\frac{c}{50}\}$. We pick $\delta>0$  such that for any
 $x,y\in\R$ with
\begin{equation}\label{eq:DefnDelta}
|e^{2\pi ix}-e^{2\pi iy}|<\delta
\end{equation}
one has that
$\inf_{a\in\Z}|x-y-a|<\epsilon$.
\subsection{ The proofs of \texorpdfstring{\eqref{eq:FlatRoof} and \eqref{eq:ConcaveRoof}}{}}
\qedsymbol \textit{ Finding $x_1$.} Consider now the sequence of sets
$$
\mathcal J'_\ell:=[a_\ell,b_\ell)\setminus (E_\ell\cup T^{-h_\ell}E_\ell),\,\ell\in\N.
$$
Pick $\ell_0\in\mathbb N$ large enough to ensure that \eqref{eq:LowerBoundForBirkhoff.1} holds and that $\frac{1}{\ell_0+1}<\frac{\delta}{4}$. Further assume that $\ell_0>19$. Observe that 
\begin{multline*}
m(\mathcal J'_{\ell_0})>(b_{\ell_0}-a_{\ell_0})-m((a_{\ell_0},b_{\ell_0})\cap E_{\ell_0})-m(T^{h_{\ell_0}}(a_{\ell_0},b_{\ell_0})\cap E_{\ell_0})\\
>(b_{\ell_0}-a_{\ell_0})\left( 1-\frac{2}{\ell_0+1}\right)\geq \frac{9(b_{\ell_0}-a_{\ell_0})}{10}.
\end{multline*}

Because $F'_{\ell_0}$ is strictly increasing on $(a_{\ell_0},b_{\ell_0})$, we can find $x_1\in \mathcal J'_{\ell_0}$ such that at least one of 
$${\rm (a)}\;m((x_1,b_{\ell_0})\cap \mathcal J'_{\ell_0})>\frac{2}{5}(b_{\ell_0}-a_{\ell_0})\text{ and }F_{\ell_0}'(x_1)>0$$
and
$${\rm(b)}\;m((a_{\ell_0},x_1)\cap \mathcal J'_{\ell_0})>\frac{2}{5}(b_{\ell_0}-a_{\ell_0})\text{ and }F_{\ell_0}'(x_1)<0$$
holds.\\
As both cases (a) and (b) are handled similarly, we assume for simplicity that $x_1$ satisfies (a). It follows that for every $x\in (x_1,b_{\ell_0})$ there is a $z_x\in (x_1,b_{\ell_0})$ with 
\begin{equation}\label{eq:LowerBoundInGoodHalf}
F_{\ell_0}(x)-F_{\ell_0}(x_1)=\underbrace{F_{\ell_0}'(x_1)(x-x_1)}_{\geq 0}+\frac{F''_{\ell_0}(z_x)}{2}(x-x_1)^2\geq \frac{c}{2}\left(\frac{x-x_1}{b_{\ell_0}-a_{\ell_0}}\right)^2.
\end{equation}
\qedsymbol \textit{ Finding $x_2$ and the proof of \eqref{eq:ConcaveRoof}.} 
Let $y\in (x_1,b_{\ell_0})$ be such that $m((x_1,y)\cap \mathcal J'_{\ell_0})=\frac{1}{5}(b_{\ell_0}-a_{\ell_0})$ and note that, by \eqref{eq:LowerBoundInGoodHalf}, 
$$F_{\ell_0}(y)-F_{\ell_0}(x_1)\geq \frac{c}{2}\left(\frac{y-x_1}{b_{\ell_0}-a_{\ell_0}}\right)^2\geq 
\frac{c}{2}\left(\frac{\frac{1}{5}(b_{\ell_0}-a_{\ell_0})}{b_{\ell_0}-a_{\ell_0}}\right)^2
\geq \frac{c}{50}\geq \epsilon.$$
Thus, there is a least $y_0\in (x_1,y]$ with $F_{\ell_0}(y_0)-F_{\ell_0}(x_1)=\epsilon$.\\ 
Because $F_{\ell_0}$ is strictly increasing and unbounded 
on the interval $(x_1,b_{\ell_0})$, we have that for any $a\in\N$, there are $r_a,s_a,t_a\in [y_0,b_{\ell_0})$, $r_a<s_a<t_a$ with 
$$
F_{\ell_0}(r_a)-F_{\ell_0}(x_1)=(a-1)+\epsilon,\,F_{\ell_0}(s_a)-F_{\ell_0}(x_1)=a-\epsilon,\text{ and }F_{\ell_0}(t_a)-F_{\ell_0}(x_1)=a+\epsilon.
$$
So, since $F_{\ell_0}$ is concave-up on  $[y_0,b_{\ell_0})$,
\begin{equation}\label{eq:ConcavityInequality}
\frac{1-2\epsilon}{s_a-r_a}=\frac{F_{\ell_0}(s_a)-F_{\ell_0}(r_a)}{s_a-r_a}\leq \frac{F_{\ell_0}(t_a)-F_{\ell_0}(r_a)}{t_a-s_a}=\frac{2\epsilon}{t_a-s_a}.
\end{equation}
Noting that $\epsilon\leq 1/4$, formula \eqref{eq:ConcavityInequality} implies that 
$$
2\epsilon(t_a-s_a)<(1-2\epsilon)(t_a-s_a)\leq 2\epsilon(s_a-r_a)
$$
and, so, 
\begin{multline*}
m\Big(\{x\in (y_0,b_{\ell_0})\,|\,\inf_{a\in\Z}|F_{\ell_0}(x)-F_{\ell_0}(x_1)-a|> \epsilon \}\Big)\\
    > \frac{1}{2} m\Big(\{x\in (y_0,b_{\ell_0})\,|\,F_{\ell_0}(x)-F_{\ell_0}(x_1)> \epsilon \}\Big)\geq \frac{b_{\ell_0}-a_{\ell_0}}{10}.
\end{multline*}
Thus,
\begin{multline*}
 m\Big(\{x\in (y_0,b_{\ell_0})\cap \mathcal J'_{\ell_0}\,|\,\inf_{a\in\Z}|F_{\ell_0}(x)-F_{\ell_0}(x_1)-a|> \epsilon \}\Big)\\
 \geq m\Big(\{x\in (y_0,b_{\ell_0})\,|\,\inf_{a\in\Z}|F_{\ell_0}(x)-F_{\ell_0}(x_1)-a|> \epsilon \}\Big)-
m((y_0,b_{\ell_0})\setminus \mathcal J'_{\ell_0})\\
>\frac{b_{\ell_0}-a_{\ell_0}}{10}-\frac{b_{\ell_0}-a_{\ell_0}}{10}=0.
\end{multline*}
Thus, there is an $x_2\in (y_0,b_{\ell_0})\cap \mathcal J'_{\ell_0}$ for which \eqref{eq:ConcaveRoof} holds.\\
\qedsymbol \textit{ The proof of formula \eqref{eq:FlatRoof}.}
Observe that ($\mathcal J_\ell$.4) together with the definition of $\mathcal J_{\ell_0}'$ imply that 
$x_j,T^{h_{\ell_0}}x_j\in \mathcal I_{\ell_0}\setminus E_{\ell_0}$ for each $j\in \{1,2\}$. Thus, by the definition of $E_{\ell_0}$,
\begin{multline*}
    |e^{2\pi i F_{\ell_0}(x_2)}-e^{2\pi i F_{\ell_0}(x_1)}|=|e^{2\pi i F_{\ell_0}(x_2)}\psi(x_2)-e^{2\pi i F_{\ell_0}(x_1)}\psi(x_2)|\\
    =|e^{2\pi i F_{\ell_0}(x_2)}\psi(x_2)+\Big(-\psi(T^{h_{\ell_0}}x_1)+e^{2\pi i F_{\ell_0}(x_1)}\psi(x_1)\Big)-e^{2\pi i F_{\ell_0}(x_1)}\psi(x_2)|\\
    \leq |\psi(T^{h_{\ell_0}}x_2)-\psi(T^{h_{\ell_0}}x_1)|+|e^{2\pi i F_{\ell_0}(x_1)}\Big(\psi(x_1)-\psi(x_2)\Big)|<\frac{4}{\ell_0+1}.
\end{multline*}
Noting that $\frac{4}{\ell_0+1}<\delta$, we see that \eqref{eq:FlatRoof} holds. We are done.
\end{proof}

\newpage

\appendix

\section{Proof of Proposition~\ref{p:psi-s-star}} \label{a:limit-trajectory}

In this section, we prove Proposition~\ref{p:psi-s-star}, which is a key technical tool used in Section~\ref{s:flows2}. 
We use the notation introduced in Sections~\ref{s:flows1} and~\ref{s:flows2}. 
Let $\gamma$ be the trajectory segment from the statement of the proposition:
$$
\gamma=\{\phi_t(x_0)\,|\,t\in[0,T]\},
\qquad x_0 \in \mm \setminus \mm_*, \; T>0;
$$
it is a closed subset of $\mm \setminus \mm_*$.
As $(\psi_s)$ is defined on the open set $\mathcal M\setminus \mm_*$, for every $p\in\mathcal M\setminus\mm_*$ there is an $\varepsilon>0$ such that for every $s\in(-\varepsilon,\varepsilon)$ the point $\psi_s(p)$ is well defined.
However, as each fixed point of $(\phi_t)$ is a singularity of $(\psi_s)$, it may happen that for some $s_1>0$ the point $\psi_{s_1}(p)$ is undefined (as we will later see, in this case $\lim_{s\rightarrow s_1^-}\psi_s(p)\in\mm_*$).
Set 
\begin{equation}\label{eq:DefnS_*}
s_* = \sup \{ s > 0 \,|\, \psi_s(\phi_{[0,T]}(x_0)) \subset \mm \setminus \mm_*\}
\end{equation}
and suppose that $s_*$ is finite, otherwise we have nothing to prove.
Consider the map $F:[s,s_*)\times [0,T]\rightarrow \mathcal M$ defined by the rule 
$$
F(s, t) = \psi_s(\phi_t(x_0)).
$$

\begin{proposition}\label{B.thm:ExtensionOfF}
The function $F:[0,s_*)\times [0,T]\rightarrow\mathcal M$ has a (unique) continuous extension $\overline F:[0,s_*]\times [0,T]\rightarrow\mathcal M$. 
\end{proposition}
\begin{proof}
For $s < s_*$, denote $\lambda_s(t) = F(s, t)$, for a fixed $s$ it is a continuous function $\lambda_s: [0, T] \mapsto \mm$.
To prove the proposition, it is enough to prove that there exists a uniform limit of those functions
\begin{equation}\label{eq:DefnOfLambda_*}
G(t)=\lim_{s\rightarrow s_*^-}\lambda_s(t).
\end{equation}
In order to check this, it is enough to prove that $\lambda_s$ is Cauchy\footnote{A standard theorem from Calculus states that a Cauchy sequence of continuous functions must uniformly converge to a continuous function. We use a version of this result where the functions have values in the metric space $\mm$, and instead of a sequence $\psi_n$ we have a parametric family $\psi_s$; the proof is almost the same.}: for every $\varepsilon>0$ there is a $\delta>0$ such that for any $t\in [0,T]$ and any $s_1,s_2\in [0,s_*)$,
\begin{equation}\label{UniformCauchyCondition}
\text{if }|s_1-s_2|<\delta,\text{ then   }d(F(s_1,t), F(s_2,t))<\varepsilon,
\end{equation}
where $d$ is the distance on $\mm$.

To show the Cauchy condition, let us take a small $\varepsilon > 0$. Set $\mm_r$ to be the set of all points in $\mm$ that are $r$-close to $\mm_*$. Denote by $v_\psi$ the vector field associated to $\psi_s$ and put     \begin{equation}\label{eq:BoundOnPsiVectorFiel}
R = \sup_{\mm \setminus \mm_{\varepsilon/4}} |v_\psi|,
\qquad
\delta = \frac{\eps}{4R}.
\end{equation}
Without loss of generality, suppose $\eps$ is small enough to ensure that  $\mm_{\eps/2}$ is a union of disjoint balls, which are at least $\eps/2$-away from each other.\\
Pick now $t\in [0, T]$ and $s_1,s_2\in [0,s_*)$ with $s_1<s_2$ and $|s_1-s_2|<\delta$. We claim that the points $p_1=F(s_1, t)$ and $p_2=F(s_2, t)$ are $\eps$-close.\\ 
Indeed, set $\beta = \{\psi_s(p_1)\,|\,s\in [0,s_2-s_1]\}$ to be the trajectory segment of $\psi_s$ connecting $p_1$ with $p_2$: indeed, $p_2=\psi_{s_2-s_1}(p_1)$. By \eqref{eq:BoundOnPsiVectorFiel}, the sum of the lengths of all the segments forming the set $\beta\setminus\mathcal M_{\varepsilon/4}$  is at most $\varepsilon/4$. We now consider the following two cases separately:
        
\noindent 1. $\beta\cap\overline{ \mathcal  M_{\varepsilon/4}}=\varnothing$. In this case, the bound on the length of $\beta$ implies that $p_1$ and $p_2$ are $\varepsilon/4$-close. 

\noindent 2. $\beta\cap\overline{ \mathcal  M_{\varepsilon/4}}\neq \varnothing$. Because any two of the balls forming $\mathcal M_{\varepsilon/4}$ are at least $\varepsilon/2$-away, the trajectory $\beta$ can only have a non-empty intersection with the closure of one of these balls; we denote the closure of this ball with the radius $\varepsilon/4$ by $B$. For $j=1,2$, we pick a "pair" $q_j\in \beta \cap B$ for the point $p_j$ in the following way: if $p_j\in B$, we set $q_j=p_j$; if $p_j\not\in B$, we take $q_j\in \beta\cap B$ to be the endpoint of the connected segment of $\beta \setminus B$ linking $p_j$ with $B$. 
We now see by applying the triangle inequality that
\[
    d(p_1,p_2)\leq d(p_1,q_1)+d(q_1,q_2)+d(q_2,p_2)
    < \varepsilon/4 + \varepsilon/2 + \varepsilon/4
    = \varepsilon,
\]
which proves the Cauchy property, and thus concludes the proof. 
\end{proof}
\begin{remark}
    Let $s_*$ and $G$ be as in formula \eqref{eq:DefnOfLambda_*}. Note that by our choice of $s_*$ there exists a $t_*\in [0,T]$ for which $\psi_{s_*}$ is not defined at the point $p=\phi_{t_*}(x_0)$. We claim that for any such $t_*$, $G(t_*)\in\mm_*$. Indeed, suppose for contradiction that $G(t_*)\not\in\mm_*$ and let $U\subset\mm\setminus \mm_*$ denote a $\psi_s$-flow box around $G(t_*)$. Noting that $\lim_{s\rightarrow s_*^-}\psi_s(p)=G(t_*)$ can hold only if the $\psi_s$-trajectory $\{\psi_s(p)\,|\,s\in [0,s_*)\}$ stops somewhere inside of $U$, we see by the continuity of $\psi_s$ in $U$ that $G(t_*)=\psi_{s_*}(p)$. A contradiction. 
\end{remark}

Let us now deduce Proposition~\ref{p:psi-s-star} from the continuity of $\overline F$. As $\gamma \subset \mm \setminus \mm_*$,
injectivity of $\phi_t: [0, T] \mapsto \gamma$ can only fail when $\gamma$ is a periodic orbit and $T$ is greater or equal to its period.  
Then we can change $T$ to be equal to the period, this will not change $\gamma$, but the injectivity assumption will be almost achieved with one exception $\phi_0(x_0) = \phi_T(x_0)$. We continue this appendix with a proof of Proposition~\ref{p:psi-s-star} for the case when injectivity fully holds. The periodic case can be easily reduced to the injective case -- for example, by dividing $\gamma$ into two halves where injectivity holds.

As $\phi_t: [0, T] \mapsto \gamma$ is a bijection, for any $x \in \gamma$ there is a unique $t(x)$ with $\phi_{t(x)}(x_0)=x$.
We can then set $\overline \psi_{s}(x) := \overline F(s, t(x))$, where $s \in [0, s_*]$ and $x \in \gamma$.
We then have a continuous map 
\[
\overline \psi_s(x) : [0, s_*] \times \gamma \mapsto \mm
\]
that coincides with $\psi_s$ when $s < s_*$.
Set $\gamma_* \subset \gamma$ to be the set of all $x$ with $\overline \psi_{s_*}(x) \in \mm_*$. 
Then, $\overline \psi_{s_*}$ coincides with $\psi_{s_*}$ on $\gamma \setminus \gamma_*$.
Note that $\gamma_*$ is not empty by the choice of $s_*$.

\begin{lemma} \label{l:very-obvious-lemma}
Let $\alpha$ be an open piece of $\gamma$ that lies in $\gamma \setminus \gamma_*$: $\alpha=\phi_{(a, b)}(x_0) \subset \gamma \setminus \gamma_*$. Suppose that $\alpha$ ends at a point $x_* = \phi_b(x_0)$ such that $x_*\in \gamma_*$.
Then, $\overline \psi_{s_*}(\alpha)$ is the final segment of a stable separatrix of a saddle, and this saddle is $\overline \psi_{s_*}(x_*)$.
\end{lemma}
\begin{proof}
We have $\overline \psi_{s_*}(\alpha) = \psi_{s_*}(\alpha)$. 
As $\frac{dH}{ds} = 1$, the map $\psi_s$ maps small pieces of trajectories of $(\varphi_t)$ into small pieces of trajectories of $(\varphi_t)$. If we have a closed trajectory piece $\beta = \varphi_{[c, d]}(x)$ such that both $\beta$ and $\psi_{s_*}(\beta)$ lie in $\mm \setminus \mm_*$, a compactness argument shows that $\psi_{s_*}(\beta)$ is a trajectory segment. As any two points on $\alpha$ can be connected by a closed trajectory segment, this means that $\psi_{s_*}(\alpha)$ is an open and connected piece of a single trajectory. 
As $\overline \psi_s$ is continuous, we have \[
\lim_{x \to x_*, \; x \in \alpha} \psi_{s_*}(x) = 
\lim_{x \to x_*, \; x \in \alpha} \overline \psi_{s_*}(x) = 
\overline \psi_{s_*}(x_*) \in \mm_*.
\]
Hence, $\overline \psi_{s_*}(\alpha)$ is the final segment of a stable separatrix, and $\overline \psi_{s_*}(x_*)$ is therefore a saddle. 
\end{proof}

\begin{corollary}
$\overline \psi_{s_*}(\gamma)$ is a fixed point or a separatrix path.    
\end{corollary}
\begin{proof}
The open set $\gamma \setminus \gamma_*$ is a union of at most countably many open trajectory pieces, which we denote by $\alpha_i$.
Each $\alpha_i$ except at most two of them is bounded by two points of $\gamma_*$. If this is the case, by Lemma~\ref{l:very-obvious-lemma} (and its symmetric version with inverse time $t$) the set $\overline \psi_{s_*}(\alpha_i)$ is a full separatrix connection between saddles. As $\psi_{s_*}$ is injective on $\gamma \setminus \gamma_*$, different $\alpha_i$ must be mapped into different separatrix connections. As there are only finitely many separatrices, we deduce that $\{\alpha_i\}$ is actually a finite set. 
Then, $\gamma_*$ is a finite union of closed intervals (possibly consisting of a single point) that we call $\beta_j$.
As $\overline \psi_{s_*}(\gamma_*) \subset \mm_*$, for each $j$ the set $\overline \psi_{s_*}(\beta_j)$ is a single point by continuity, and this point is a fixed point of $(\phi_t)$.
If $\gamma_* = \gamma$, then $\overline \psi_{s_*}(\gamma)$ is a single fixed point. If $\gamma_* \ne \gamma$, then each interval $\beta_j$ has an adjacent interval $\alpha_i$, and thus all $\beta_j$ must be mapped into saddles of $(\phi_t)$.
The intervals $\alpha_i$ between them are mapped into separatrices connecting them, thus forming a separatrix path.     
\end{proof}

\noindent The corollary above implies Proposition~\ref{p:psi-s-star}, because when $s \to s_*^{-}$, we have $\psi_s(\gamma) = \overline \psi_s(\gamma) \to \overline \psi_{s_*}(\gamma)$ by the continuity of $\overline \psi_s$.

\section{Proof of Lemma~\ref{l:integral-computation}} \label{s:integrals-puiseux}

We first need to prove the following auxiliary lemma.
\begin{lemma} \label{l:integral-auxiliary}
    Let $f(x, y)$ be an analytic function, and let $a(x) < b(x)$ be analytic functions defined in a neighborhood of zero. 
    Then for any integer $n>0$ the integral
    \[
    I(h) =  \int_{\{0 < x < h^\frac 1 n, \; a(x) < y < b(x)\}} f(x, y) \; dx dy
    \]
    is given on $(0, h_0)$ by a converging \lp series, where $h_0 > 0$ is small enough.
\end{lemma}
\begin{proof}
Take any analytic antiderivative $F(x, y)$ of $f$ w.r.t. $y$: $\frac{\partial}{\partial y}F = f$.
Integrating over $y$ gives
\[
I(h) = \int_0^{h^{1/n}} F(b(x)) - F(a(x)) \; dx.
\]
As the integrand is analytic, it has an analytic antiderivative $G(x)$, and so $I(h) = G(h^{1/n}) - G(0)$ is a \lp series (without any logarithms).
\end{proof}

\begin{proof}[Proof of Lemma~\ref{l:integral-computation}]
The case $m=0$ is covered by Lemma~\ref{l:integral-auxiliary} by taking $a(x) = 0$ and $b(x) = Y(x)$; the inequality $x < X(y)$ is redundant for small enough $h$. So we now assume $m > 0$. Take $\delta > 0$ that is less then $Y(0)$ and $X(0)$. Then for small enough $h$ the domain of integration
\[
x^n y^m < h, \; 0 < x < X(y), \; 0 < y < Y(x)
\]
splits into three parts:
\begin{enumerate}
    \item $x^n y^m < h, \; \delta \le y < Y(x), \; 0 < x < \delta$,
    \item $x^n y^m < h, \; \delta \le x < X(y), \; 0 < y < \delta$, 
    \item $x^n y^m < h, \; 0 < x, y < \delta$.
\end{enumerate}
The integral of $f$ over the first part can be reduced to Lemma~\ref{l:integral-auxiliary} by the coordinate change $x_1 = y^{m/n}x$, $y_1 = y$, which is analytic because $y \ge \delta$ and transforms the inequality $x^n y^m < h$ into $x_1^n < h$. Similarly, the integral over the second part can be dealt with by the coordinate change $x_1 = x$, $y_1 = x^{n/m}y$. It remains to consider the integral $I_3$ over the third part, for which we apply the linear scaling $x = \delta \tilde x$, $y = \delta \tilde y$: 
\[
I_3(h) = \int_{\{x^n y^m < h, \; 0 < x, y < \delta\}} f(x, y) \; dx dy
= \delta^2 \int_{\{\tilde x^n \tilde y^m < \tilde h, \; 0 < \tilde x, \tilde y < 1\}} g(\tilde x, \tilde y) \; d\tilde x d \tilde y,
\]
where $\tilde h = \delta^{-n-m} h$ and $g(\tilde x, \tilde y) = f(\delta \tilde x, \delta \tilde y)$.

Decomposing $g(\tilde x, \tilde y) = \sum_{k, l} \alpha_{k, l} \tilde x^k \tilde y^l$ into Taylor series\footnote{If the series does not converge on the whole domain, we decrease $\delta$}, we get $I_3 = \delta^2 \sum_{k, l} \alpha_{k, l} B_{k, l}$ with 
\[
B_{k, l}(\tilde h) = \int_{\{\tilde x^n \tilde y^m < \tilde h, \; 0 < \tilde x, \tilde y < 1\}} \tilde x^k \tilde y^l \; d\tilde x d\tilde y.
\]
Let $\tilde x_*(\tilde h, \tilde y) = \tilde h^{\frac 1 n}\tilde y^{- \frac m n}$ denote the solution of $\tilde x_*^n \tilde y^m = \tilde h$ and $\tilde y_*(\tilde h) = \tilde h^{\frac 1 m}$ denote the solution of $\tilde x_*(\tilde h, \tilde y_*) = 1$.
We have
\begin{align*}
\begin{split}
B_{k, l}(\tilde h) &= 
\int_{0 < \tilde x, \tilde y < 1, \; \tilde x^n \tilde y^m < \tilde h} \tilde x^k \tilde y^l \; d\tilde x d\tilde y
= \tilde y_* + \int_{\tilde y_*}^1 \left( \tilde y^l \int_0^{\tilde x_*} \tilde x^k  \; d\tilde x \right) d\tilde y \\
& = \tilde h^{\frac 1 m} + \frac{1}{k+1} \int_{\tilde y_*}^1 \tilde y^l \tilde h^{\frac {k+1}{n}} \tilde y^{-\frac{m(k+1)}{n}} d\tilde y\\
& = \tilde h^{\frac 1 m} +  
\frac{1}{(k+1)(l+1-\frac{m(k+1)}{n})} 
\tilde h^{\frac {k+1}{n}} (1 - \tilde y_*^{l+1-\frac{m(k+1)}{n}}) \\
& = \tilde h^{\frac 1 m}
+\frac{1}{(k+1)(l+1-\frac{m(k+1)}{n})} \left(
    \tilde h^{\frac {k+1}{n}} - \tilde h^{\frac{l+1}{m}}
\right),  
\end{split}    
\end{align*}
assuming $\frac {k+1} n \ne \frac{l+1} m$. If $\frac {k+1} n = \frac{l+1} m$, we get
\[
\tilde J_{k, l} = \tilde h^{\frac 1 m} + \frac{1}{m(k+1)}\tilde h^{\frac{k+1}n} |\ln \tilde h|.
\]
Hence, $I_3(h)$ is given by a convergent \lp series.
\end{proof}

\textit{Acknowledgments.} The authors would like to thank Bassam Fayad for several discussions on the topic.
Rigoberto Zelada is supported by EPSRC through Joel Moreira's Frontier Research Guarantee grant, ref. EP/Y014030/1 and would like to thank the Brin Mathematics Research Center for partially funding his visit to the Universisty of Maryland on  2025  when a large portion of the  current paper was completed.

Adam Kanigowski\\
\textsc{Department of Mathematics,  University of Maryland, College Park, MD 20742, USA} and  \textsc{Faculty of Mathematics and Computer Science, Jagiellonian University, Lojasiewicza 6, Krakow, Poland}\par\nopagebreak
\noindent
\href{mailto:akanigow@umd.edu}
{\texttt{akanigow@umd.edu}}

\medskip 

\noindent
 Alexey Okunev\\
\textsc{Department of Mathematics,  University of Maryland, College Park, MD 20742, USA}\par\nopagebreak
\noindent
\href{mailto:aokunev@umd.edu}
{\texttt{aokunev@umd.edu}}

\medskip

\noindent
Rigoberto Zelada\\
\textsc{Mathematics Institute. University of Warwick, Coventry, CV4 7AL, UK}\par\nopagebreak
\noindent
\href{mailto:rigoberto.zelada@warwick.ac.uk}
{\texttt{rigoberto.zelada@warwick.ac.uk}}


\begin{thebibliography}{99}


\bibitem{AGV} V.\ Arnold, S.\ Gusein-Zade, A.\ Varchenko,
{\em Singularities of Differentiable Maps: Volume II Monodromy and Asymptotic Integrals}, Monographs in Mathematics 83, Birkhauser (1988).

\bibitem{Ati} M. Atiyah, {\em Resolution of singularities and division of distributions}, Communications on pure and applied mathematics 23.2 (1970): 145-150.
%
 \bibitem{AF} A. Avila, G. Forni, {\em Weak mixing for interval exchange transformations and translation flows}, Ann. of Math. 165 (2007), 637--664.
\bibitem{ChFKU} J.\ Chaika, K.\ Fraczek, A.\ Kanigowski, C.\ Ulcigrai,
{\em Singularity of the spectrum for smooth area-preserving flows in genus two and translation surfaces well approximated by cylinders}, Comm. Math. Phys. 381 (2021), no.3, 1369–1407.
%
\bibitem{ChW} J.\ Chaika, A.\ Wright, {\em A smooth mixing flow on a surface with non-degenerate fixed points}, J. Amer. Math. Soc. 32 (2019), 81--117.
%
\bibitem{CFr} J-P. Conze, K. Fraczek,
{\em Cocycles over interval exchange transformations and multivalued Hamiltonian flows}, Adv. Math. 226 (2011), no. 5, 4373--4428.
%
\bibitem{CornfeldErgodicTheory}
 I. P. Cornfeld, S. V.\ Fomin,  Ya. G.\ Sina\u i,
{\em Ergodic theory}, Grundlehren der mathematischen Wissenschaften ,
    245, (1982), Springer-Verlag, New York. Translated from the Russian by A. B. Sosinski\u i
%
 \bibitem{FKZ} B. Fayad, A. Kanigowski, R. Zelada, {\em A non-mixing Arnold flow on a surface}, Adv. Math. 465 (2025), 1-54
%
\bibitem{FFK} B. Fayad, G. Forni, A. Kanigowski, {\em Lebesgue spectrum of countable multiplicity for conservative flows on
the torus}, J. Am. Math. Soc. 34 (2021), 747--813.
%
\bibitem{FaKa} B. Fayad and A. Kanigowski, {\em On multiple mixing for a class of conservative surface flows},
Inv. Math. 203 (2) (2016), 555--614.
%
\bibitem{FKU}K. Fraczek, A. Kanigowski, C. Ulcigrai,  {\em Singularity of the spectrum of typical minimal smooth area-preserving flows in any genus}, arXiv:2505.13193
%
\bibitem{Green} M. Greenblatt, {\em Resolution of singularities, asymptotic expansions of integrals and related phenomena}, Journal d'Analyse Mathématique 111.1 (2010): 221-245.
%
\bibitem{Hir} H. Hironaka, {\em Resolution of singularities of an algebraic variety over a field of characteristic zero}, Ann.
of Math., Vol. 79, 1964, pp. 109-326.
%
\bibitem{Jean} P. Jeanquartier, {\em Developpement asymptotique de la distribution de Dirac attache une fonction analytique}, C. R. Acad. Sci. Paris, Ser. A-B 201 (1970), A1159–A1161.
%
\bibitem{KLU} A. Kanigowski, M. Lema\'nczyk, C. Ulcigrai, {\em On disjointness properties of some parabolic flows}, Invent. Math.
221 (2020), 1-111.
%
\bibitem{Kat} A.\ Katok, {\em Interval exchange transformations and some special flows are not mixing}, Israel J. Math. 35 (1980), 301--310.
%
\bibitem{Kat73} A. Katok, {\em Invariant measures of flows on oriented surfaces}, Doklady Akademii Nauk, Russian Academy of Sciences, 211 4 (1973), 775--778.
%
\bibitem{Koch1} A. V. Kochergin, {\em Mixing in special flows over a shifting of segments and in smooth flows on surfaces},
Mat. Sb., 96 138 (1975), 471--502.
 \bibitem{Koch2} A. V. Kochergin, {\em Non-degenerate fixed points and mixing in flows on a 2-torus}, Matematicheskii
Sbornik, 1948, (2003), 83-112 (Translated in: Sb. Math., 194:1195-1224).
%
\bibitem{Kocc}A.\ Kochergin, {\em Well-aproximable angles and mixing for flows on $\T^2$ with nonsingular fixed points},
 Electronic research announcements of the American Mathematical Society, Vol. 10, 113--121 (October 26, 2004)
 \bibitem{Kolg} A. N. Kolmogorov, {\em On dynamical systems with an integral invariant on the torus}, Doklady Akad. Nauk SSSR 93 (1953), 763--766
\bibitem{Lev} G. Levitt, {\em Feuillettages des surfaces}, thesis, 1983.
%
\bibitem{Loes} F. Loeser, {\em Volume de tubes autour de singularities}, Duke Math. J. 53 (1986), 443–455.
%
\bibitem{Mai}A. Mayer, {\em Trajectories on the closed orientable surfaces}, Rec. Math. [Mat. Sbornik] N.S., 12(54) (1943), 71--84.
\bibitem{NZ} I.\ Nikolaev,  E.\ Zhuzhoma, {\em Flows on 2-dimensional manifolds}, Lecture Notes in Mathematics, 1705, (1999), Springer-Verlag, New York.
\bibitem{Ravotti} D. Ravotti, {\em Quantitative mixing for locally Hamiltonian flows with saddle loops on compact surfaces}, Ann.
H. Poincar\'e 18 (2017), 3815--3861.
%
\bibitem{KS} Ya.G. Sinai, K.M. Khanin, {\em Mixing for some classes of special flows over rotations of the circle}, Funktsionalnyi Analiz i Ego Prilozheniya, 26 3 (1992):1 21 (Translated in: Functional Analysis and its Applications, 26:3:155-169, 1992).
%
\bibitem{Sklo} M. D. Shklover, {\em On dynamical systems on the torus with continuous spectrum}, Izv. Vuzov 10 (1967), 113--124.
%
\bibitem{Ulc2} C. Ulcigrai, Mixing of asymmetric logarithmic suspension flows over interval exchange transformations, Ergod. Th. Dyn. Sys. 27 (2007), 991--1035.
%
\bibitem{Ulc} C.\ Ulcigrai, {\em Weak mixing for logarithmic flows over interval exchange transformations}, J.
Mod. Dynam. 3 (2009),35--49.

\bibitem{Ulc1} C. Ulcigrai, {\em Absence of mixing in area-preserving flows on surfaces}, (2011) Annals of Mathematics 173 (2011),1743--1778

%
\bibitem{viana2006} M. Viana, {\em Ergodic theory of interval exchange maps},  Revista Matem{\'a}tica Complutense 19(1) (2006), 7--100.
%
\bibitem{Zor94} A. Zorich, {\em Hamiltonian flows of multivalued hamiltonians on closed orientable surfaces}, Preprint, Max-Planck Institut fur Mathematik, Bonn (1994)
%
\bibitem{Zor} A. Zorich, {\em How do the leaves of a closed 1-form wind around a surface?}, in Pseudoperiodic topology, vol. 197 of Amer.Math. Soc. Transl. Ser. 2, Amer. Math. Soc., Providence, RI, 1999,135--178.
\end{thebibliography}
\end{document}